\documentclass[10pt]{amsart}
\usepackage[utf8]{inputenc}
\usepackage[a4paper,left=3cm, right= 3cm, top=3cm, bottom= 3cm]{geometry}
\usepackage{amsmath}
\usepackage{amsfonts}
\usepackage{amssymb}
\usepackage{amsthm}
\usepackage{bbold}
\usepackage{color}
\usepackage{enumerate}
\usepackage{amsmath, amssymb, amstext,amsthm}
\usepackage[toc,page]{appendix}
\usepackage{color}

\usepackage{xcolor}

\usepackage[active]{srcltx}
\usepackage{ifpdf}
\ifpdf 
    \usepackage[pdftex,     
            plainpages=false,   
            breaklinks=true,    
            colorlinks=true,
            pdftitle=My Document
            pdfauthor=My Good Self
           ]{hyperref} 
\else 
    \usepackage{hyperref}       
\fi

\DeclareMathOperator{\dv}{div}
\newcommand{\tor}{\mathbb{T}^2}
\newcommand{\T}{\mathbb{T}^d}

\DeclareMathOperator{\supp}{supp}

\theoremstyle{plain}
\newtheorem{thm}{Theorem}[section]
\newtheorem{lemma}[thm]{Lemma}
\newtheorem{prop}[thm]{Proposition}

\newtheorem{cor}[thm]{Corollary}

\theoremstyle{definition}
\newtheorem{dfn}[thm]{Definition}
\newtheorem{rem}[thm]{Remark}

\newcommand{\dx}{\mathrm{d}x}

\newcommand{\ds}{\mathrm{d}s}

\newcommand{\di}{\mathrm{d}x_1}
\newcommand{\dii}{\mathrm{d}x_2}

\newcommand{\curl}{\operatorname{curl}}

\newcommand{\sumk}{\sum_{k=1}^4}
\newcommand{\tme}{\operatorname{time}}
\newcommand{\quadr}{\operatorname{quad}}
\newcommand{\Rtr}{\overset{\circ}{R_0}}

\def\R{\mathbb{R}}

\def\N{\mathbb{N}}
\def\Z{\mathbb{Z}}

\def\T{\mathbb{T}}

\def\div{\operatorname{div}}
\def\curl{\operatorname{curl}}

\def\supp{\mathrm{supp\, }}

\title[Non-uniqueness for 2D Euler in Hardy spaces]{Non-uniqueness and energy dissipation for 2D Euler equations with vorticity in Hardy spaces}
\author{Miriam Buck, Stefano Modena}
\email{mbuck@mathematik.tu-darmstadt.de, stefano.modena@gssi.it}
\address{\noindent Technische Universit\"at Darmstadt, Fachbereich Mathematik, D-64285 Darmstadt, Germany \newline Gran Sasso Science Institute, 67100 L'Aquila, Italy}

\date{\today}                     

\begin{document}

\begin{abstract}
We construct by convex integration examples of energy dissipating solutions to the 2D Euler equations on $\R^2$ with vorticity in the Hardy space $H^p(\R^2)$, for any $2/3<p<1$.
\end{abstract}

\thanks{The second named author would like to thank Jan Burczak for several interesting discussions on the topics considered in this paper.}

\maketitle

\section{Introduction}
In this paper we consider the 2-dimensional incompressible Euler equations on the full space $\R^2$
\begin{align}\label{2D Euler}
\begin{cases}
\partial_t u + \dv(u\otimes u) + \nabla p = 0,\\
\dv u = 0,\\
u(\cdot, 0) = u_0,
\end{cases}
\end{align}
where $u:\R^2\times [0,1]\rightarrow\R^2$ is the velocity field of some fluid and $p:\R^2\times [0,1]\rightarrow\R$ is the corresponding (scalar) pressure. 

It is well known that the system \eqref{2D Euler} is globally well posed in $W^{s,2}$ for $s>2$, in the sense that for initial data $u_0 \in W^{s,2}$ there is a \emph{unique} solution $u \in C([0,1], W^{s,2}(\R^2))$ defined on the whole time interval $[0,1]$ (more precisely on the whole time half-line $[0, + \infty)$). 

It is however of fundamental importance, both mathematically and physically, to understand what happens in case of ``rougher'' initial data, and in particular if it is still possible, in case of rougher initial data, to prove existence and uniqueness of (weak) solutions to \eqref{2D Euler}. 

\subsection{Short literature overview}

The starting point of this analysis is the observation that \eqref{2D Euler} can be formally rewritten as a transport equation for the vorticity $\omega = \curl u$ via
\begin{align}\label{voriticty eq}
\begin{cases}
\partial_t\omega + u\cdot \nabla\omega = 0,\\
u= \nabla^\perp\Delta^{-1}\omega.
\end{cases}
\end{align}
From \eqref{voriticty eq} it is clear that the $L^p$ norm of the vorticity of any smooth solution to \eqref{2D Euler} is conserved in time, for any $p\in [1,\infty]$. In the framework of weak solutions, it is thus natural to ask the following question:\\

\textbf{Q1}: \noindent \emph{For $u_0\in L^2(\R^2)$ with $\curl u_0\in L^1(\R^2) \cap L^p(\R^2)$ for some $p\in [1,\infty]$, does there exist a unique solution $u\in C([0,1],L^2(\R^2))$ to \eqref{2D Euler} with $\curl u \in C([0,1],L^1(\R^2) \cap L^p(\R^2))$ and initial datum $u_0$?}

\smallskip

\noindent or, more generally, 

\smallskip

\textbf{Q2}: \noindent \emph{For $u_0\in L^2(\R^2)$ with $\curl u_0\in X$ for some Banach space $X$, does there exist a unique solution $u\in C([0,1],L^2(\R^2))$ to \eqref{2D Euler} with $\curl u \in C([0,1],X)$ and initial datum $u_0$?}\\

The first result in this direction is due to Yudovich \cite{yudovich1962some,yudovich1963non} for the case $p=\infty$ and it states that for any initial datum $u_0\in L^2$ with $\omega_0\in L^1\cap L^\infty$, there exists a unique global solution $u\in C_tL^2_x$ with $\omega\in L^\infty_t(L^1_x\cap L^\infty_x)$ to \eqref{voriticty eq}. Yudovich result is based on the observation that even though a bounded vorticity $\omega$ does not imply Lipschitz bound on the velocity field $u$ (hence the classical ``smooth'' theory can not be simply applied), nevertheless it is possible to deduce $\log$-Lipschitz bounds on $u$, which are enough to show well posedness. 

For $p<\infty,$ the question turns out to be much more delicate (and still open in its generality to this date): indeed, an $L^p$ bound on $\omega$ implies, in general, only bounds on $u$  in some $C^\alpha$ space of H\"older continuous functions, and this is in general not enough to apply Yudovich techniques and show well-posedness of \eqref{voriticty eq} (some partial extension of Yudovich's result appeared in \cite{loeper2006uniqueness}, where functions with vorticity in $\bigcap_{p<\infty} L^p$ were considered, with strong bounds  on the growth of $L^p$ norms as $p \to \infty$). 

There have been however in the last years several important results, providing partial answers to questions {\bf Q1} and {\bf Q2} above.
We mention few of them, and in particular those concerning the problem of non-uniqueness of weak solutions.

In \cite{vishik2018instability1,vishik2018instability2} Vishik gave a negative answer to {\bf Q1}, proving nonuniqueness in the class of solutions having vorticity $\omega \in L^\infty_t(L^p_x)$, however not for the Euler system \eqref{2D Euler} (or \eqref{voriticty eq}), but for the Euler system \eqref{2D Euler} with a $L^1_t (L^1_x \cap L^p_x)$ external force  (thus allowing for an additional ``degree of freedom''). Vishik's proof is based on a careful analysis of the linearized operator $\mathcal{L}$ associated to \eqref{2D Euler} and on the construction of an \emph{unstable} eigenvalue for $\mathcal{L}$. 

Another approach based on numerical simulations has been proposed by Bressan and Shen in \cite{bressan2020posteriori}, where an initial profile is constructed for which there is numerical evidence of non-uniqueness, but a rigorous proof of this result is still missing. 

Very recently, in \cite{mengual2023non}, Mengual proved  that for any $2<p<\infty$ there exists initial data $u_0\in L^2(\R^2)$ with initial vorticity $\curl u_0\in L^1\cap L^p$ for which there are infinitely many admissible solutions $u\in C_tL^2$ to \eqref{2D Euler} but with the drawback that $\curl u(t, \cdot)$ does not belong to $L^p(\R^2)$ for any $t>0$. An admissible solution is a weak solution that does not increase the kinematic energy, i.e. $\frac{1}{2}\|u(t)\|^2_{L^2}\leq \frac{1}{2}\|u(0)\|^2_{L^2}$ for a.e. $t$.


Concerning the more general question {\bf Q2}, Bruè and Colombo address this question in \cite{brue2021nonuniqueness} for the case that $X$ is the Lorentz space $X=L^{1,\infty}$.  They construct a sequence $(u_n)_n$ of smooth ``approximate'' solutions to \eqref{2D Euler}, converging to an ``anomalous'' weak solution $u$ of \eqref{2D Euler} (in the sense that $u$ is nonzero, but $u|_{t=0} = 0$, thus providing an example of non-uniqueness) and having the additional property that the sequence of vorticities $(\curl u_n)_n$ is a Cauchy sequence in $L^{1,\infty}$.

The construction in \cite{brue2021nonuniqueness} is based on an \emph{intermittent convex integration scheme}. As we shall explain in Section \ref{ssec_our_result} below, it is expected that, in general, intermittent convex integration schemes in dimension $d$ can provide (``anomalous'') weak solutions to the Euler equations having vorticity in $L^p$ only if
\begin{equation}
\label{eq_lp_vorticity}
p < \frac{2d}{d+2}
\end{equation}
In particular, in dimension $d=2$, it is not possible with the current techniques to construct solutions $u$ with $\curl u \in L^p$, not even for $p=1$. This motivated the authors in \cite{brue2021nonuniqueness} to look for velocity fields with vorticity in $L^{1, \infty}$, a function space which is ``weaker'' than $L^1$ in terms of integrability, but which scales as $L^1$. 

It has however to be noted that, as we mentioned before, the result in \cite{brue2021nonuniqueness} shows the existence of a sequence $\{u_n\}_n$ of approximate solutions to \eqref{2D Euler} converging strongly in $L^2$ to an anomalous weak solution $u$ to \eqref{2D Euler} and whose corresponding vorticities $\{\curl u_n\}_n$ build a Cauchy sequence in $L^{1,\infty}$ which thus has a limit $\omega$ in $L^{1, \infty}$. However, since $L^{1,\infty}$ is not a space of distributions (precisely, it does not embed into $\mathcal{D}'$),  it is not clear whether and in what sense the distributional vorticity of the solution $u$ (or, in other words, the distributional limit of $\curl u_n$) coincide with the $L^{1,\infty}$ limit $\omega$. 

Indeed, in general, there is no connection between distributional limit and limit in $L^{1,\infty}$. Standard examples where this absence of connection can be explicitly seen can be constructed even in one dimension, see, for instance, Section \ref{ssec_counterexample} below, where a sequence $(f_n)$ of  piecewise constant maps is constructed, with $f_n$ converging to two very different ``objects'' in distributions and in $L^{1,\infty}$ respectively: a Dirac delta in $\mathcal{D}'$ and the zero function in $L^{1,\infty}$. Similar constructions can also be done for smooth $(f_n)$.

\subsection{Our result}
\label{ssec_our_result}

The result by Bru\`e and Colombo \cite{brue2021nonuniqueness} motivated us to see if the methods used in \cite{brue2021nonuniqueness} could be adapted to show non-uniqueness of weak solutions to \eqref{2D Euler} with vorticity in some other function space $X$ that is ``weaker'' than $L^1$ in terms of integrability, but at the same time it does embed into $\mathcal{D}'$, avoiding the issues connected to the $L^{1, \infty}$ topology. 

The real Hardy spaces $H^p$ for $p<1$ (thus matching with \eqref{eq_lp_vorticity} in dimension $d=2$) turns out to be a natural choice, as $H^p$ does embed into $\mathcal{D}'$ for any $p \in (0,\infty)$ (see  Definition \ref{def: hardy space} for the precise definition of the space $H^p$). Precisely, we prove the following theorem.
\begin{thm}[Main Theorem]
\label{thm:main}
Let $\frac{2}{3}<p<1$. For any energy profile $e\in C^\infty\left([0,1];\left[\frac{1}{2},1\right]\right)$ there exists a solution $u\in C([0,1],L^2(\R^2))$ to \eqref{2D Euler} with
\begin{enumerate}[(i)]
\item $\int_{\R^2} |u|^2(t) \,\dx = e(t)$,
\item $\curl u \in C([0,1],H^p(\R^2))$.
\end{enumerate}
In particular, there exist energy dissipating solutions $u \in C_t L^2_x$ to \eqref{2D Euler} with $\curl u \in C_t H^p_x$. 
 
Furthermore, for energy profiles $e_1,e_2$ such that $e_1=e_2$ on $[0,t_0]$ for some $t_0\in [0,1]$, there exist two solutions $u_1,u_2$ satisfying $(i)$, $(ii)$ with $u_1(t)=u_2(t)$ for $t\in [0,t_0]$.
\end{thm}

\begin{cor}
Let $\frac{2}{3}<p<1$. There are two admissible (in the sense that the total kinetic energy is non-increasing in time) solutions $u_1, u_2 \in C([0,1]; L^2(\R^2))$ with $\curl u_1, \curl u_2 \in C([0,1]; H^p(\R^2))$ with the same initial datum $u_1|_{t=0} = u_2|_{t=0}$.
\end{cor}
\begin{proof}
The proof follows immediately from Theorem \ref{thm:main}, picking two non-increasing energy profiles $e_1, e_2$ which coincide on $[0,1/2]$ and are different from each other on $[1/2,1]$. 
\end{proof}

\begin{rem}
We add some remarks about the statement of Theorem \ref{thm:main}. 
\begin{enumerate}
\item Differently from typical results in convex integration, we work on the full space $\R^2$ and not on the periodic domain $\T^2$. This is motivated by the fact that Hardy spaces are usually defined and studied on the full space and it is quite hard to find references for Hardy spaces on $\T^2$ (or $\T^d$). This creates some technical troubles we are going to discuss in Section \ref{ssec_technical_nov}.
\item The constraint $p>2/3$ comes exactly from the fact that we are working on the full space, and it tells essentially that the objects we can construct are not decaying \emph{too} fast at $\infty$. We expect this constraint can be removed if one works in Hardy spaces on $\T^2$.
\item Differently than in \cite{brue2021nonuniqueness}, condition (ii) in the statement of Theorem \ref{thm:main} means precisely that the \emph{distributional} $\curl$ of $u(t)$ belongs to $H^p$, for all $t$ (with continuous dependence on time). 
\end{enumerate}
\end{rem}

We wish now to spend some words in explaining why conditions \eqref{eq_lp_vorticity} plays a fundamental role (both in \cite{brue2021nonuniqueness} and in our result), and therefore why we were able to show Theorem \ref{thm:main} only under the condition $p<1$. 

As in \cite{brue2021nonuniqueness}, we use a convex integration technique in the spirit of De Lellis and Székelyhidi works on the 3D Euler equations in the framework of Onsager's Theorem (see \cite{de2009euler,de2014dissipative,de2013dissipative,isett2018proof,buckmaster2018}). The outline in all of these schemes is an iterative construction where, starting from an initial approximate solution, one adds fast oscillating perturbations with a higher frequency $\lambda_n \to \infty$ with respect to the typical frequencies $\lambda_{n-1}$ in the previous approximation. In case of the Euler equation, given an approximate solution $(u_{n-1},p_{n-1},R_{n-1})$ with error term on the right hand side
\begin{align}\label{eq: error previous}
\partial_t u_{n-1} + \dv(u_{n-1}\otimes u_{n-1}) + \nabla p_{n-1}= - \dv R_{n-1},
\end{align}
one makes the Ansatz
\begin{align*}
u_{n}(t,x)&= u_{n-1}(t,x) + w_n(t,x) + \text{lower order corrector terms}
\end{align*}
with
\begin{align*}
w_n(t,x) &= a_{n-1}(t,x) W_{\lambda_n},\\
W_{\lambda_n}(x) = W(\lambda_n x)&: \text{ fast oscillating building block},\\
a_{n-1}&: \text{ slowly varying coefficient}, \quad a_{n-1} \approx |R_{n-1}|^{1/2}.
\end{align*}
The interaction of $w_n$ (having frequencies $\lambda_n$) with itself from the nonlinearity of the equation produces a term having frequencies $\approx \lambda_{n-1}$ and it allows therefore for the cancellation of the previous error, provided
\begin{align*}
a_{n-1}^2\int_{\tor}W_{\lambda_n}\otimes W_{\lambda_n}\,\dx \, \approx R_{n-1} \int_{\tor}W_{\lambda_n}\otimes W_{\lambda_n}\,\dx \, \sim\, R_{n-1}.
\end{align*}
In particular, this forces us to choose a building block $W$ such that
\begin{align}\label{scaling w}
\int_{\tor} W \otimes W \,\dx = \int_{\tor} W_{\lambda_n} \otimes W_{\lambda_n} \,\dx \, \sim \, 1,
\end{align}
which in turn implies (taking the trace in the above relations) that
\begin{equation}
\label{eq_l2_norm_w}
\|W\|_{L^2}^2 = \|W_{\lambda_n}\|_{L^2}^2 \, \sim \, 1. 
\end{equation}
Clearly, since $W_{\lambda_n}$ is fast oscillating with frequency $\lambda_n \gg 1$ one expects very little control on the first  derivative of $W_{\lambda_n}$ (and thus also on $\curl u_n$). In particular, one can not expect that  $\|\nabla W_{\lambda_n}\|_{L^\infty}$ or even  $\|\nabla W_{\lambda_n}\|_{L^2}$ stays bounded as $n \to \infty$. 

There is however some hope in controlling $\|\nabla W_{\lambda_n}\|_{L^p}$ if $p \ll 2$, or, more precisely, if \eqref{eq_lp_vorticity} holds. Indeed, for those $p$'s for which \eqref{eq_lp_vorticity} does not holds, we have the embedding $W^{1,p} \hookrightarrow L^2$ and thus \eqref{eq_l2_norm_w} combined with the Sobolev inequality gives
\begin{equation*}
1 \, \sim \, \|W\|_{L^2}^2 \leq \|\nabla W\|_{L^p}
\end{equation*}
so that there is no hope in showing smallness of $\|\nabla W\|_{L^p}$. On the other side, if \eqref{eq_lp_vorticity} holds, the Sobolev inequality fails and thus it is possible to construct a sequence of building blocks $W_{\lambda_n}$ oscillating with frequencies $\lambda_n$, satisfying \eqref{eq_l2_norm_w} and, at the same time, having
\begin{equation*}
\|\nabla W_{\lambda_n}\|_{L^p} \to 0 \text{ as } n \to \infty. 
\end{equation*}

This was the crucial observation of Buckmaster and Vicol in the groundbreaking work \cite{buckmaster2019nonuniqueness}, where the authors apply a convex integration scheme to the Navier-Stokes equations and need therefore to control higher order derivatives of the perturbation, because of the presence of the dissipative term in the system. Similar observations were used also in \cite{modena2018non}, \cite{modena2019non}, \cite{modena2020convex}, \cite{brue2021positive} \cite{cheskidov2021nonuniqueness}, \cite{cheskidov2022extreme}, \cite{giri2021non}, \cite{pitcho2021almost} for constructing counterexamples to uniqueness for the transport equations with Sobolev vector fields and other more recent works (see e.g. \cite{burczak-mod-sze21}, \cite{novack1}, \cite{giri20233},\cite{giri2023wavelet}).

As we observed before, in dimension $d=2$, condition \eqref{eq_lp_vorticity} corresponds to $p<1$, hence preventing the possibility of estimating $\curl u$ in $L^1$ with the current techniques. On the other side, the key observation in \cite{brue2021nonuniqueness} is that  
for the Lorentz space $L^{1,\infty}$, the Sobolev embedding fails,
 \begin{align*}
 \|\nabla u\|_{L^{1,\infty}}\nleqslant \|u\|_{L^2} \text{ in general, for } u \in C^\infty(\T^2),
 \end{align*}
and this made the construction in \cite{brue2021nonuniqueness} possible.

If one were allowed to choose $p<1$ in \eqref{eq_lp_vorticity}, the embedding 
\begin{align*}
\|\nabla u\|_{L^{p}}\nleqslant \|u\|_{L^2}
\end{align*}
would also fail. Even though $L^p$ spaces are defined also for $p<1$, they do not embed continuously into $\mathcal{D}'$, hence a construction with vorticity in $L^p$ for $p<1$ would suffer from the same issues as the construction in Lorentz spaces. 

It turns however out that a feasible subsitute for $L^p$ in the range $p \in (0, \infty)$ is the Hardy space $H^p$. Indeed, on one hand, we have $H^p(\R^2)\cong L^p(\R^2)$ for $p>1$ and $H^p(\R^2)\subset L^1(\R^2)$ for $p=1$. On the other hand, $H^p$ embeds into $\mathcal{D}'$ for all $p \in (0, \infty)$ and, finally, functions in $H^p$ scale like $L^p$ (also for $p<1$), in the sense that
 \begin{align}\label{scaling hp}
 \|\nabla^l \varphi(\mu\cdot)\|_{H^p} = \mu^{l-\frac{2}{p}}\| \nabla^l \varphi\|_{L^p(\R^2)}
 \end{align}
for $\varphi\in C^\infty_c(\R^2)$ and any $0<p<\infty$, so that one can hope to have  a sequence of building blocks which have $L^2$ norm of order $1$ (as in \eqref{eq_l2_norm_w}) and, at the same time, having vorticity with $H^p$ norm arbitrarily small, if $p<1$.

\subsection{Technical novelties}
\label{ssec_technical_nov}

We briefly explain now the two main technical novelties of this paper compared to previous works on convex integration. They concern
\begin{enumerate}
\item how elements in Hardy spaces can be estimated and, in particular, how to exploit the scaling properties \eqref{scaling hp} in Hardy spaces;
\item how to do the construction on the full space, where also decay at $\infty$ has to be taken into account.
\end{enumerate}

\subsubsection{Concentration in Hardy spaces}
As we mentioned before, in order to control the quantity $\|\curl w\|_{H^p}$, we use the mechanism of \emph{concentration} or \emph{intermittency} that was also used in \cite{brue2021nonuniqueness} for the control of the norm in $L^{1,\infty}$.
The building blocks are defined via concentrated functions,
\begin{align*}
W := W_\mu :=\varphi_\mu(x) \xi
\end{align*}
where $\varphi\in C^\infty_c(\R^2)$, $\varphi_\mu$ is the periodization of the concentrated function $\mu\varphi(\mu x)$ and $\xi\in\R^2$ is some given direction. The scaling is such that we keep \eqref{eq_l2_norm_w}, i.e.
$\|W_\mu\|_{L^2} \,\raisebox{-0.9ex}{\~{}}\, 1$.  The main problem in exploiting concentration in the framework of Hardy spaces (with $p<1$) is that there is no Hölder inequality available: in general
\begin{align*}
\|a f\|_{H^p} \nleqslant \|a\|_{L^\infty} \|f\|_{H^p}.
\end{align*}
Hence the estimate for $\|\curl w\|_{H^p}$ is more subtle and we cannot use \eqref{scaling hp} directly. 

To deal with this issue, one could use the definition of Hardy norm (see \eqref{eq_def_hardy_norm}), but this turns out to be extremely difficult. We use therefore the notion of \emph{atoms}, which are typical functions $f$ in Hardy space that have support in a ball $B$ and satisfy the cancellation property $\int_B f \,\dx = 0$ and an $L^\infty$ estimate, see Definition \ref{dfn: hardy atoms}. Indeed, thanks to the intermittency, one can view the perturbation $w(x) = \chi_{\kappa_0}(x) a(x) W(x)$ as a finite sum of functions, each of them supported on a very small ball of radius $\frac{1}{\mu}$, i.e.
\begin{align*}
w &= \sum \theta_j,\\
\theta_j &=  \mathbb{1}_{B_{\frac{1}{\mu}}(x_j)} w
\end{align*}
for some $x_1, \dots, x_n$. The curl of each $\theta_j$ satisfies the cancellation property $\int_{B_{\frac{1}{\mu}}(x_j)} \curl \theta_j \,\dx = 0$ as a derivative of a compactly supported function. Therefore, $\curl w$ is a linear combination of atoms and thus $\curl w\in H^p$. One can use a standard estimate for atoms (see Lemma \ref{lemma: hardy atoms}) on each $\theta_j$, balancing $\|\theta_j\|_{L^\infty}$ (estimated by \eqref{scaling hp}) and the size of its support.

\subsubsection{Full spaces vs periodic domain}

Since we are constructing solutions in $L^2(\R^2)$ and not in $L^2(\T^d)$, we need to implement a convex integration scheme that differs from previous ones in at least two more ways:
\begin{enumerate}[(i)]
\item As fast oscillating perturbations are used to reduce the error, $\tor$ is the natural habitat for solutions constructed by convex integration schemes. We want to keep the advantages from using fast oscillations, while also ensuring the decay at infinity.
\item On a more technical side, there is no bounded right inverse $\dv^{-1}:L^1(\R^2;\R^2)\rightarrow L^1(\R^2;\mathcal{S})$ (here $\mathcal{S}$ is the space of symmetric matrices) for the divergence. In order to reduce $R_0$, it is crucial to construct an antidivergence  for functions of the form $f u_\lambda$ with $f\in C_c^\infty(\R^2), u\in C^\infty_0(\tor)$ that takes advantage of the oscillation with an estimate of the form $\|\dv^{-1} (fu_\lambda)\|_{L^1} \approx \frac{1}{\lambda} \|fu\|_{L^1}$.
\end{enumerate}
We deal with (i) by using that if $R_0\in L^1(\R^2)$
\begin{align*}
\lim_{\kappa\rightarrow\infty} \|R_0\|_{L^1(\R^2\setminus\ B_\kappa)} = 0
\end{align*}
and reduce the error only on a compact set $B_{\kappa_0}$ such that $\|R_0\|_{L^1(\R^2\setminus B_{\kappa_0})}\ll 1,$ using a cutoff $\chi_{\kappa_0}$ in our perturbations
\begin{align*}
w(t,x) &= \chi_{\kappa_0}(x) a(t,x) W_\lambda(x).
\end{align*}
Therefore, the support of $w$ consists of a (possibly very large) finite number (which is of order $\kappa_0^2$) of periodic boxes of the form $[0,1]^2+k$ for some $k\in\Z^2$ that is fixed at the start of each iteration. This allows us to have similar estimates as for periodic functions on $\tor$ with a factor depending on $\kappa_0,$ while also having perturbations in $L^2(\R^2).$

Concerning (ii), we gain the factor $\frac{1}{\lambda}$ by using integration by parts: On $\tor$, we have the bounded (in $L^1$) operator $\dv^{-1}: C^\infty_0(\tor;\R^2)\rightarrow C^\infty_0(\tor;\mathcal{S})$ that satisfies $\|\dv^{-1} u_\lambda\|_{L^1(\tor)} \leq \frac{C}{\lambda}\|u\|_{L^1(\tor)}$ (see Lemma \ref{standard antidivergence} below or also, for instance, \cite[Proposition 4]{burczak-mod-sze21}). Defining
\begin{align*}
R_1(f,u_\lambda) = f \dv^{-1}u_\lambda,
\end{align*}
we have $\|R_1\|_{L^1}\leq \frac{C(\supp f)}{\lambda} \|f\|_{C(\R^2)} \|u\|_{L^1(\tor)}$ and this matrix satisfies
\begin{align*}
\dv R_1 = fu_\lambda + (\dv^{-1}u_\lambda)\cdot\nabla f .
\end{align*}
Since $\div^{-1}$ is not bounded from $L^1(\R^2) \to L^1(\R^2)$, we can not write the last term as a divergence of a tensor field whose $L^1$ norm is bounded by the $L^1$ norm of $(\dv^{-1}u_\lambda)\cdot \nabla f$. Hence we simply set
\begin{align*}
r_1 = -(\dv^{-1}u_\lambda)\cdot\nabla f
\end{align*}
so that 
\begin{align}
\label{eq_intro_antidiv}
r_1 + \dv R_1 = fu_\lambda \text{ and } \|R_1\|_{L^1},  \|r_1\|_{L^1} \lesssim \frac{1}{\lambda}.
\end{align}
We therefore work with approximate solutions that satisfy
\begin{align*}
\partial_t u_{n-1} + \dv(u_{n-1}\otimes u_{n-1}) + \nabla p_{n-1}= - r_{n-1} -  \dv R_{n-1}
\end{align*}
instead of \eqref{eq: error previous}. In order to cancel this additional error term, we include in our definition of $u_n$ a corrector of the form
\begin{align*}
v(x,t) = \int_0^t r_{n-1}(x,s)\,\ds
\end{align*}
such that $\partial_t v - r_{n-1} = 0.$ Since in this way $r_{n-1}$ enters into the definition of the perturbation through $v$, we have to make sure to control $\|\curl \int_0^t r_{n-1}(x,s)\,\ds\|_{H^p}$. We do this by carrying out the ``integration by parts'' $N$ times, yielding $(r_N,R_N)$ with
\begin{align*}
r_N + \div R_N = f u_\lambda  \text{ and } \|\nabla r_N\|_{L^\infty} \lesssim \frac{1}{\lambda^{N-1}}
\end{align*} 
instead of \eqref{eq_intro_antidiv}. We then make sure that $r_N$ has compact support, so that we can again use the standard estimate for atoms mentioned above ($L^\infty$ bound together with a bound on the size of the support) .

\subsection{An explicit example comparing distributional and Lorentz space convergence}
\label{ssec_counterexample}

We conclude this introduction with an example of a sequence $(f_n)_n$ of $1D$ piecewise constant maps (but similar constructions can be done with smooth maps) converging to \emph{different limits} in $L^{1, \infty}$ and in $\mathcal{D}'$.  In particular, $f_n \to \delta_0$ in distributions, whereas $f_n \to 0$ in $L^{1, \infty}$. 
Set
\begin{align*}
f_n = \frac{1}{n}\sum_{j=0}^{n-1} 2^{n+j} \mathbb{1}_{[2^{-(n+j)}, 2^{-(n+j-1)}]}
\end{align*}  
Then $\int_\R f_n \,\dx = 1$ and it is not difficult to see that
\begin{align*}
f_n\rightarrow \delta_0 \text{ in } \mathcal{D}'(\R).
\end{align*}
On the other hand, it holds
\begin{align*}
\left|\lbrace |n f_n|\geq t\rbrace\right| = \begin{cases} 2^{-n},   &t\in (0,2^{n}],\\
2^{-(n+1)}, &t\in (2^{n}, 2^{n+1}],\\
\vdots\\
2^{2n-1}, &t\in (2^{2n-2} ,2^{2n-1}],\\
0, & t> 2^{2n-1}.
\end{cases}
\end{align*}
This yields
\begin{align*}
\sup_{t>0} t\left|\lbrace |n f_n|\geq t\rbrace\right| \leq 1,
\end{align*}
and therefore 
\begin{align*}
\|f_n\|_{L^{1,\infty}(\R)} \rightarrow 0 \text{ for } n\rightarrow\infty.
\end{align*}

\subsection{Notation}

We fix some notation we are going to use in the paper.
\begin{itemize}
\item We denote by $e_1, e_2$ the standard basis vectors of $\R^2$.
\item For any vector $\xi=(\xi_1,\xi_2)\in\R^2$, we will denote by $\xi^\perp$ the orthogonal vector $\xi^\perp = (\xi_2, - \xi_1)$.
\item We denote by $\mathcal{S}$ the set of symmetric $2\times 2$ matrices.
\item For a quadratic $2\times 2$ matrix $T$, we denote by $\overset{\circ}{T} = T-\frac{1}{2} \operatorname{tr}T \operatorname{Id}$ its traceless part.
\item For a function $f\in C^1(\R^2)$ we denote by $\nabla^\perp f = (\partial_2 f,-\partial_1 f)$ its orthogonal gradient.
\item For $d_1,d_2\in\N$ we write $f:\mathbb{T}^{d_1}\rightarrow\R^{d_2}$ for a function $f:\R^{d_1}\rightarrow\R^{d_2}$ defined on the full space that is periodic with period $1$ in all variables, i.e. $f(x+l e_k) = f(x)$ for all $k=1,\dots, d_1,$ $l\in\Z$.
\item For a periodic function $f$ as above, we denote $\int_{\mathbb{T}^{d_1}} f \,\dx = \int_{[0,1]^{d_1}} f\,\dx$, i.e. the integral over just one periodic box.
\item $C_0^\infty(\tor;\R^d)=\lbrace f:\tor\rightarrow\R^d \text{ smooth}, \int_{\tor} f\,\dx = 0\rbrace$ is the space of smooth periodic functions on $\R^2$ with zero mean value on one periodic box.
\item For a function $g\in C^\infty(\tor)$ and $\lambda\in\mathbb{N},$ we denote by $g_\lambda:\tor\rightarrow\mathbb{R}$ the $\frac{1}{\lambda}$ periodic function
$$g_\lambda (x) := g(\lambda x).$$
Notice that for every $l\in\mathbb{N}$, $s\in [1,\infty]$
$$\|D^l g_\lambda\|_{L^s(\tor)} = \lambda^l \|D^l g\|_{L^s(\tor)}.$$
\item $H^p(\R^2)$ is the real Hardy space, see Definition \ref{def: hardy space}.
\item $L^2_\sigma(\R^2) = \left\lbrace f\in L^2(\R^2): \dv f = 0 \text{ in distributions}\right\rbrace$ is the space of divergence-free vector fields in $L^2(\R^2)$.
\item For a function $f:[0,1]\times\R^2\rightarrow\R^d$ and $s\in[1,\infty]$, we write $\|\cdot\|_{C_tL^s_x}$ for the norm $\|f\|_{C_tL^s_x} = \max_{t\in[0,1]}\|f(t)\|_{L^s(\R^2)}$.
\item For any function $\varphi:\R\rightarrow\R$ with $\supp(\varphi)\subset (-\frac{1}{2},\frac{1}{2})$ and $\mu>1$ we write $\varphi_\mu$ for the periodic extension of the function $\mu^\frac{1}{2} \varphi(\mu \left(x-\frac{1}{2}\right))$, whose support is contained in intervalls of length $\frac{1}{\mu}$ centered around the points $\frac{1}{2}+\Z$. Note that
\begin{align}\label{eq: scaling}
\|\varphi_\mu\|_{L^r(\mathbb{T})}=\mu^{\frac{1}{2} - \frac{1}{r}}\|\varphi\|_{L^r(\R)}
\end{align}
 and in particular $\|\varphi_\mu\|_{L^2(\mathbb{T})}=\|\varphi\|_{L^2(\R)}$.
\item Let $\lambda\in\N$, $f:\tor\rightarrow\R^d$. We will sometimes write $f_\lambda$ for the oscillating functions $f_\lambda(x) = f(\lambda x)$. On the other hand, for $f:\R\rightarrow\R$ with compact support, we will oftentimes write $f_\mu$ for its concentrated version. To avoid confusion, we will only use the parameter $\lambda$ for oscillations and $\mu$ (or $\mu_1,\mu_2$, respectively) for concentration.
\end{itemize}

\section{Preliminaries}
We now provide the technicals tools that are needed for the proof of the Main Theorem \ref{thm:main} and we start this section with two useful estimates for functions of the form  $fg_\lambda$, where \mbox{$f\in C_c^\infty(\R^2),$}  $g\in C^\infty(\tor).$ For these estimates it is crucial that $f$ is compactly supported. Note also that the size of $\supp f$ enters the estimate.
\begin{prop}[Improved Hölder]\label{prop: improved hoelder}
Let  $k,\lambda\in\mathbb{N},$ $f:\R^2\rightarrow\R$ smooth with $\supp f\subset [-k,k]^2$ and $g:\mathbb{T}^2\rightarrow\R$ smooth. Then it holds for all $s\in [1,\infty]$
\begin{align*}
\|fg_\lambda\|_{L^s(\R^2)}\leq \|f\|_{L^s(\R^2)}\|g\|_{L^s(\mathbb{T}^2)} + \frac{C(s)(2k)^\frac{2}{s}}{\lambda^\frac{1}{s}}\|f\|_{C^1(\R^2)}\|g\|_{L^s(\mathbb{T}^2)}.
\end{align*}
\end{prop}

\begin{proof}
This is an adaptation of Lemma 2.1 in \cite{modena2018non}, which can be proven in the same way.
\end{proof}

\begin{lemma}\label{lemma: fast oscillations}
Let  $k,\lambda\in\mathbb{N},$ $f:\R^2\rightarrow\R$ smooth with $\supp f\subset [-k,k]^2$ and $g:\mathbb{T}^2\rightarrow\R$ smooth with $\int_{\tor} g \,\dx = 0$. Then
\begin{align*}
\left|\int_{[-k,k]^2} f(x) g_\lambda(x)\,\dx\right|\leq \frac{4\sqrt{2}k^2\|f\|_{C^1(\R^2)} \|g\|_{L^1(\tor)}}{\lambda}.
\end{align*}
\end{lemma}

\begin{proof}
This is an adaptation of Lemma 2.6 in \cite{modena2018non} with the same proof.
\end{proof}

\begin{dfn}[Hardy spaces on $\R^2$]\label{def: hardy space}
 Let $\Psi\in \mathcal{S}(\R^2)$ be a Schwartz function with $\int_{\R^2}\Psi(x)\,\dx \neq 0$ and let $\Psi_\varepsilon(x) = \frac{1}{\varepsilon^2}\Psi(\frac{x}{\varepsilon})$. For any $f\in \mathcal{S}'(\R^2),$ we define the  radial maximal function
\begin{align}
\label{eq_def_hardy_norm}
m_\Psi f(x) = \sup_{\zeta>0}|f\ast\Psi_\zeta(x)|.
\end{align}
Let $0<p<\infty$. The real Hardy space $H^p(\R^2)$ is defined as the space of tempered distributions
\begin{align*}
H^p(\R^2) = \left\lbrace f\in \mathcal{S}'(\R^2) : m_\Psi f \in L^p(\R^2)\right\rbrace
\end{align*}
and we write
\begin{align*}
\|f\|_{H^p(\R^2)} = \|m_\Psi f\|_{L^p(\R^2)}.
\end{align*}
Note that $\|\cdot\|_{H^p(\R^2)}$ is only a quasinorm. 
The definition of $H^p(\R^2)$ does not depend on the choice of the function $\Psi$ and the quasinorms are equivalent. For $p>1$, the space $H^p(\R^2)$ coincides with the Lebesgue space $L^p(\R^2)$. For $p\leq 1$, $H^p(\R^2)$ is a complete metric space with the metric given by $d(f,g)=\|f-g\|^p_{H^p(\R^2)}$ and the inclusion $H^p(\R^2)\hookrightarrow \mathcal{S}'(\R^2)$ is continuous.
\end{dfn}

\begin{dfn}[Hardy space atoms]\label{dfn: hardy atoms}
For $p\leq 1$, a Hardy space atom is a measurable function $a$ with the following properties:
\begin{enumerate}[(i)]
\item $\supp a\subset B$ for some ball $B$,
\item $|a|\leq |B|^{-\frac{1}{p}}$
\item $\int_B x^\beta a(x)\, \dx = 0$ for all multiindices $\beta$ with $|\beta|\leq 2(p^{-1}-1)$.
\end{enumerate}
\end{dfn}

\begin{lemma}[Estimate for Hardy space atoms]\label{lemma: hardy atoms}
There is a uniform constant $C$ such that for all atoms $a$ it holds
\begin{align*}
\|a\|_{H^p(\R^2)}\leq C.
\end{align*}
\end{lemma}
\begin{proof}
We refer to \cite{stein1993harmonic}, see 2.2 in Chapter III.2.
\end{proof}

\begin{rem}\label{rem: hardy atoms}
\begin{enumerate}
\item We will use that for a function $f$ satisfying $(iii)$ in Definition \ref{dfn: hardy atoms} with support in a ball $B$, we have by \mbox{Lemma \ref{lemma: hardy atoms}}
\begin{align*}
\|f\|_{H^p(\R^2)} \leq C |B|^\frac{1}{p}\|f\|_{L^\infty(\R^2)}.
\end{align*}
\item Since $\frac{2}{3}<p<1$ in our case, we only need to check the $0$th momentum in $(iii)$, i.e. $\int_B a(x)\, \dx = 0$.
\end{enumerate}
\end{rem}

\begin{lemma}[Standard antidivergence]\label{standard antidivergence}
There exists a linear operator $$\dv^{-1}: C^\infty_0(\tor;\R^2)\rightarrow C^\infty_0(\tor;\mathcal{S})$$ such that $\dv \dv^{-1} u =u$ and
\begin{align*}
\|\nabla^l\dv^{-1} u\|_{L^s(\tor)}&\leq C(s)\|\nabla^lu\|_{L^s(\tor)} ,\\
\|\nabla^l\dv^{-1} u_\lambda\|_{L^s(\tor)}&\leq \frac{C(s)}{\lambda^{1-l}}\|\nabla^l u\|_{L^s(\tor)}\text{ for all } l,\lambda\in\N,s\in [1,\infty].
\end{align*}
\end{lemma}
For the proof see Proposition 4 in \cite{burczak-mod-sze21}.

For $N\geq 2$ we inductively define
\begin{align*}
\dv^{-N} u = \sum_{k=1,2} \dv^{-1}\left(\dv^{N-1}u\cdot e_k\right).
\end{align*}

With that standard antidivergence operator, we will define an improved antidivergence operator for functions of the form $fu_\lambda$, $f\in C^\infty_c(\R^2)$, $u\in C^\infty_0(\tor;\R^2)$, on the full space. 
\begin{lemma}[Improved antidivergence operators]\label{antidivergence}
\begin{enumerate}[(i)]
\item For any $N\in\N$, there exists a bilinear operator $$S_N: C_c^\infty(\R^2;\R)\times C_0^\infty(\tor;\R^2)\rightarrow C_c^\infty(\R^2;\R^2)\times  C_c^\infty(\R^2;\mathcal{S})$$ such that for $S_N(f,u) = (r,R)$ it holds
\begin{align*}
r + \dv R = fu
\end{align*}
with
\begin{align*}
\|\nabla^l r\|_{L^\infty(\R^2)}&\leq C(\supp f)\|\nabla^l\dv^{-N} u\|_{L^\infty(\tor)} \|f\|_{C^{N+l}(\R^2)}\text{ for all } l\in\N,\\
\|R\|_{L^1(\R^2)} &\leq C(\supp f) \|\dv^{-1}u\|_{L^1(\tor)}\|f\|_{C^{N-1}(\R^2)}.
\end{align*}
\item For any $N\in\N$, there exists a bilinear operator $$\tilde{S}_N: C_c^\infty(\R^2;\R^2)\times C_0^\infty(\tor;\mathcal{S})\rightarrow C_c^\infty(\R^2;\R^2)\times  C_c^\infty(\R^2;\mathcal{S})$$ such that for $\tilde{S}_N(f,T) = (r,R)$ it holds
\begin{align*}
r + \dv R = Tf
\end{align*}
with
\begin{align*}
\|\nabla^l r\|_{L^\infty(\R^2)}&\leq C(\supp f)\|\nabla^l\dv^{-N} T\|_{L^\infty(\tor)} \|f\|_{C^{N+l}(\R^2)} \text{ for all } l\in\N,\\
\|R\|_{L^1(\R^2)} &\leq C(\supp f) \|\dv^{-1}T\|_{L^1(\tor)}\|f\|_{C^{N-1}(\R^2)}. 
\end{align*}
where, by a slight abuse of notation, we define
\begin{align*}
\dv^{-N} T= \sum_{k=1,2} \dv^{-N}(Te_k).
\end{align*}
\end{enumerate}
\end{lemma}
\begin{proof}
Let us inductively define
\begin{align*}
r_0:C_c^\infty(\R^2;\R)\times C_0^\infty(\tor;\R^2)&\rightarrow C_c^\infty(\R^2;\R^2),\\
r_0(f,u)&= fu,\\
R_0:C_c^\infty(\R^2;\R)\times C_0^\infty(\tor;\R^2)&\rightarrow C_c^\infty(\R^2;\mathcal{S}),\\
R_0(f,u) &= 0
\end{align*}
and for $N\geq 1$
\begin{align*}
r_N:C_c^\infty(\R^2;\R)\times C_0^\infty(\tor;\R^2)&\rightarrow C_c^\infty(\R^2;\R^2),\\
r_N(f,u)&= -\sum_{k=1,2} r_{N-1}(\partial_kf,\dv^{-1}u\cdot e_k),\\
R_N:C_c^\infty(\R^2;\R)\times C_0^\infty(\tor;\R^2)&\rightarrow C_c^\infty(\R^2;\mathcal{S}),\\
R_N(f,u) &= f\dv^{-1}u -\sum_{k=1,2}R_{N-1}(\partial_kf,\dv^{-1}u\cdot e_k).
\end{align*}
It is clear that
\begin{align*}
r_0(f,u) + \dv R_0(f,u) = fu.
\end{align*}
Let us assume that
\begin{align*}
r_N(f,u) + \dv R_N(f,u) = fu 
\end{align*}
for some $N\in\N$ for all $f\in C_c^\infty(\R^2)$, $u\in C_0^\infty(\tor;\R^2)$.
Then we also have
\begin{align*}
r_{N+1}(f,u) + \dv R_{N+1}(f,u) &=\sum_{k=1,2} r_N(\partial_k f,\dv^{-1}u\cdot e_k)\\
&\hspace{0,3cm} +  \dv\left(f\dv^{-1}u - \sum_{k=1,2} R_N(\partial_k f, \dv^{-1}u \cdot e_k)\right)\\
&= fu + (\dv^{-1}u)\cdot \nabla f\\
&\hspace{0,3cm} -\sum_{k=1,2} r_N(\partial_k f,\dv^{-1}u\cdot e_k) - \dv\left(\sum_{k=1,2} R_N(\partial_k f, \dv^{-1}u \cdot e_k)\right)\\
&= fu + (\dv^{-1}u)\cdot \nabla f  - \sum_{k=1,2} \partial_k f \dv^{-1}u\cdot e_k = fu.
\end{align*}
Therefore, we set
\begin{align*}
S_N(f,u)=(r_N(f,u),R_N(f,u)).
\end{align*}
For the second operator, we simply set for $f\in C_c^\infty(\R^2;\R^2)$, $T\in C_0^\infty(\tor;\mathcal{S})$
\begin{align*}
\tilde{S}_N(f,T) = \sum_{k=1,2} S_N(f_k, Te_k).
\end{align*}
The estimates follow directly from the ones for $\dv^{-1}$ from Lemma \ref{standard antidivergence}.
\end{proof}

\begin{rem}\label{rem: improved antidiv}
In particular, if $(r_N,R_N)= S_N(f,u_\lambda)$, then
\begin{align*}
\|\nabla^l r_N\|_{L^\infty(\R^2)}&\leq \frac{C(\supp f)}{\lambda^{N-l}}\|\nabla^l u\|_{L^\infty(\tor)} \|f\|_{C^{N+l}(\R^2)}\text{ for all } l\in\N,\\
\|R_N\|_{L^1(\R^2)} &\leq \frac{C(\supp f)}{\lambda} \|u\|_{L^1(\tor)}\|f\|_{C^{N-1}(\R^2)}
\end{align*}
and the same holds for $\tilde{S}_N$.
\end{rem}

\begin{lemma}[A helpful computation]\label{lemma: div of matrix}
Let $f,g\in C^1(\R)$. For any vector $\xi\neq 0\in \R^2$ it holds
\begin{align*}
\dv\left( f(\xi\cdot x) g(\xi^\perp\cdot x) \frac{\xi}{|\xi|}\otimes  \frac{\xi}{|\xi|}\right) &= f'(\xi\cdot x) g(\xi^\perp\cdot x) \xi,\\
\dv\left(f(\xi\cdot x) g(\xi^\perp\cdot x) \frac{\xi^\perp}{|\xi|}\otimes \frac{\xi}{|\xi|} \right) &= f'(\xi\cdot x) g(\xi^\perp\cdot x)\xi^\perp.
\end{align*}
\end{lemma}
\begin{proof}
The proof is trivial.
\end{proof}

\begin{dfn}\label{def: antidivergence by hand}
For $\psi_1$, $\psi_2$, $\Psi\in C^1(\R)$ with $\Psi''= \psi_2$ and a vector $\xi\neq 0$ we define
\begin{align*}
A(\psi_1,\psi_2,\xi) &= \psi_1(\xi\cdot x) \Psi'(\xi^\perp\cdot x)\left(\frac{\xi}{|\xi|}\otimes \frac{\xi^\perp}{|\xi|} + \frac{\xi^\perp}{|\xi|}\otimes \frac{\xi}{|\xi|}\right) \\
&\hspace{0,3cm} - \psi_1'(\xi\cdot x) \Psi(\xi^\perp\cdot x)  \frac{\xi^\perp}{|\xi|}\otimes\frac{\xi^\perp}{|\xi|}.
\end{align*}
and
\begin{align*}
B(\psi_1,\psi_2,\xi) &=\psi_1(\xi\cdot x) \Psi'(\xi^\perp\cdot x) \frac{\xi^\perp}{|\xi|}\otimes\frac{\xi^\perp}{|\xi|} .
\end{align*}
By Lemma \ref{lemma: div of matrix}, these symmetric matrices satisfy
\begin{align*}
\dv A &= \psi_1(\xi\cdot x)\psi_2(\xi^\perp\cdot x) \xi,\\
\dv B &= \psi_1(\xi\cdot x)\psi_2(\xi^\perp\cdot x) \xi^\perp.
\end{align*}
Let $\mu_2\gg\mu_1$. It is not difficult to see that for $\psi_1,\psi_2,\Psi\in C^\infty_c(\R)$ with zero mean value and $\Psi''=\psi_2$, supported in $(-\frac{1}{2},\frac{1}{2})$, we have for their concentrated, fast oscillating extensions
\begin{align*}
A(\psi_{1,\mu_1}(\lambda\cdot),\psi_{2,\mu_2}(\lambda\cdot),\xi)\in C_0^\infty(\tor,\mathcal{S}),\\
B(\psi_{1,\mu_1}(\lambda\cdot),\psi_{2,\mu_2}(\lambda\cdot),\xi)\in C_0^\infty(\tor,\mathcal{S})
\end{align*}
if $\xi\in\N^2$ and the estimates
\begin{align}
\|\nabla^l A(\psi_{1,\mu_1}(\lambda\cdot),\psi_{2,\mu_2}(\lambda\cdot),\xi)\|_{L^s(\tor)}\leq  \lambda^{l-1}\mu_1^{\frac{1}{2}-\frac{1}{s}}\mu_2^{l-\frac{1}{2}-\frac{1}{s}}\max_{j_1,j_2=0,1}\|\psi_1^{(j_1)}\|_{L^s(\mathbb{T})}\|\Psi^{(j_2)}\|_{L^s(\mathbb{T})},\nonumber\\
\|\nabla^l B(\psi_{1,\mu_1}(\lambda\cdot),\psi_{2,\mu_2}(\lambda\cdot),\xi)\|_{L^s(\tor)}\leq  \lambda^{l-1}\mu_1^{\frac{1}{2}-\frac{1}{s}}\mu_2^{l-\frac{1}{2}-\frac{1}{s}}\max_{j_1,j_2=0,1}\|\psi_1^{(j_1)}\|_{L^s(\mathbb{T})}\|\Psi^{(j_2)}\|_{L^s(\mathbb{T})},\label{est A B}
\end{align}
where one uses $\mu_2\gg\mu_1$. 
\end{dfn}

\section{Main Proposition}
In this section we present the main proposition that is the key to prove Theorem \ref{thm:main}. To this end, we first introduce the Reynolds defect equation:
\begin{dfn}[Solution to the Reynolds defect equation]
A solution to the Reynolds-defect-equation is a tuple $(u,p,R,r)$ of smooth functions
\begin{align*}
u\in C([0,1],L^2(\R^2)\cap L^3(\R^2)),p\in C([0,1],L^2(\R^2)), R\in C([0,1],L^1(\R^2;\mathcal{S})),\\
 r\in C([0,1],L^\infty(\R^2)), \supp_{(t,x)} r\subseteq [0,1]\times\R^2 \text{ compact},
\end{align*} such that
\begin{align*}
\partial_t u + \dv(u\otimes u) + \nabla p &=-r -\dv \overset{\circ}{R},\\
\dv u &= 0
\end{align*}
is satisfied.
\end{dfn}

\begin{prop}[Main Proposition]\label{prop: main proposition}
Let $e\in C^\infty\left([0,1];\left[\frac{1}{2},1\right]\right)$ be an arbitrary given energy profile. There exists a constant $M_0>0$ such that the following holds:  Choose $\delta, \eta>0$ with
\begin{align*}
0<\delta<1, 0<\eta<\frac{1}{32}\delta,
\end{align*} and assume that there exists a solution $(u_0, R_0, r_0,p_0)$ to the Reynolds-Defect-equation, satisfying
\begin{align}
\frac{3}{4}\delta e(t)\leq e(t)-\int_{\R^2}|u_0|(x,t)^2\,\dx\leq\frac{5}{4}\delta e(t),\label{eq: assumption energy}\\
40 \|R_0\|_{C_tL^1_x} + \|r_0\|_{C_tL^2_x} + 2\|u_0(t)\|_{L^2(\R^2)}\|r_0\|_{C_tL^2_x}\leq \frac{1}{32}\delta. \label{eq: assumption r}
\end{align}
Then there exists another solution $(u_1, R_1, r_1,p_1)$ such that 
\begin{enumerate}[(i)]
\item \begin{align*}
\frac{3}{8}\delta e(t)\leq e(t)-\int_{\R^2}|u_1|^2(x,t)\,\dx\leq\frac{5}{8}\delta e(t),
\end{align*}
\item $r_1$ satisfies $$ \|r_1\|_{C_tL^2_x} + \|u_1\|_{C_tL^2_x} \|r_1\|_{C_tL^2_x} \leq \eta$$ and
\item $$ \|\int_0^t\curl r_1(s)\,\ds\|^p_{H^p(\R^2)}\leq\eta,$$
\item $\|R_1(t)\|_{L^1(\R^2)}\leq \eta + 4\|r_0\|_{C_tL^2_x} +2 \|r_0\|_{C_tL^2_x} \|u_0(t)\|_{L^2(\R^2)}$,
\item $\|u_1(t)-u_0(t)\|_{L^2(\R^2)}\leq M_0 \delta^\frac{1}{2}$,
\item $\|\curl (u_1-u_0)(t)\|^p_{H^p(\R^2)}\leq \eta + \|\int_0^t\curl r_0(s)\,\ds\|^p_{H^p(\R^n)}$.
\end{enumerate}
\end{prop}
\begin{proof}[Proof of the Main Theorem assuming Proposition \ref{prop: main proposition}]
The solution to \eqref{2D Euler} is constructed iteratively. We start with the trivial solution $(u_0,p_0,R_0,r_0)\equiv 0$ and choose $\delta_0 = 1$. Then obviously \eqref{eq: assumption energy} and \eqref{eq: assumption r} are satisfied.  Let  $\delta_n= 2^{-n}$ for $n\geq 0$ and $\eta_n= \frac{\delta_{n+1}}{11584}$ for $n\geq -1$. Assuming that the first $n+1$ solutions $(u_{j},p_{j},R_{j},r_{j})_{0\leq j\leq n}$ are already constructed and that $(u_n,p_n,R_n,r_n)$ satisfies \eqref{eq: assumption energy}, \eqref{eq: assumption r} with $\delta_n$,  we obtain $(u_{n+1},p_{n+1},R_{n+1},r_{n+1})$ by applying Proposition \ref{prop: main proposition} with $\delta_n$, $\eta_n$. 
We show that we can proceed the iteration, i.e. that $(u_{n+1},p_{n+1},R_{n+1},r_{n+1})$  satisfies  \eqref{eq: assumption energy}, \eqref{eq: assumption r} with $\delta_{n+1}$. First, we note that by $(ii)$, we have 
\begin{align}\label{inductive est for r}
\|r_j\|_{C_tL^2_x} + \|u_j\|_{C_tL^2_x}\|r_j\|_{C_tL^2_x} \leq \eta_{j-1}
\end{align}
for all $0\leq j\leq n+1$. Now, by $(i)$, the new solution satisfies
\begin{align*}
\frac{3}{8}\delta_n e(t)\leq e(t)-\int_{\R^2}|u_{n+1}(t)|^2\,\dx\leq\frac{5}{8}\delta_n e(t)
\end{align*}
and therefore
\begin{align*}
\frac{3}{4}\delta_{n+1} e(t)\leq e(t)-\int_{\R^2}|u_1(t)|^2\,\dx\leq\frac{5}{4}\delta_{n+1} e(t),
\end{align*}
i.e. \eqref{eq: assumption energy} is satisfied. Also, by $(iv)$ and \eqref{inductive est for r} we have
\begin{align*}
&40 \|R_{n+1}\|_{C_tL^1_x} + \|r_{n+1}\|_{C_tL^2_x} + 2\|u_{n+1}(t)\|_{L^2(\R^2)}\|r_{n+1}\|_{C_tL^2_x}\\
&\hspace{0,3cm}\leq 40(\eta_n + 4\|r_n\|_{C_tL^2_x} + 2\|r_n\|_{C_tL^2_x}\|u_n\|_{C_tL^2_x})\\
&\hspace{0,3cm} +  \|r_{n+1}\|_{C_tL^2_x} + 2\|u_{n+1}(t)\|_{L^2(\R^2)}\|r_{n+1} \|_{C_tL^2_x}\\
&\hspace{0,3cm}\leq 40\eta_n + 160\eta_{n-1} + 2\eta_n = 362\eta_n = \frac{1}{32}\delta_{n+1},
\end{align*}
hence \eqref{eq: assumption r} holds. This shows that with our choice of $(\delta_n)_{n}$ and $(\eta_n)_{n}$ we can indeed construct a sequence $(u_n,p_n,R_n,r_n)_{n\in\N}$ of solutions to the Reynolds-defect-equation. By $(v)$, 
\begin{align*}
\sup_{t\in [0,1]}\|u_{n+1}(t)-u_n(t)\|_{L^2(\R^2)}\leq M_0 2^{-\frac{n}{2}}
\end{align*}
for all $n\in\N$, i.e. there exists $u\in C([0,1],L^2_\sigma(\R^2))$ such that $u_n\rightarrow u$ in $C([0,1],L^2_\sigma(\R^2))$. By $(ii)$ and  $(iv)$,
\begin{align*}
r_n&\rightarrow 0 \text{ in } C([0,1], L^1(\R^2)),\\
R_n,\overset{\circ}{R}_n &\rightarrow 0 \text{ in } C([0,1], L^1(\R^2,\mathcal{S})),
\end{align*}
showing that $u$ is a weak solution to \eqref{2D Euler}. By $(iii)$ and $(vi)$, inductively we have
\begin{align*}
\|\curl (u_{n+1}-u_n)(t)\|^p_{H^p(\R^2)}\leq \eta_n + \eta_{n-1}, 
\end{align*}
which shows that there exists $v\in C([0,1],H^p(\R^2))$ such that 
\begin{align*}
\curl u_n \rightarrow v \text{ in } C([0,1],H^p(\R^2)).
\end{align*}
But since $ H^p(\R^2)\hookrightarrow \mathcal{S}'(\R^2)$ is a continuous inclusion, this shows that $v = \curl u$. 
\end{proof}

\section{The Building Blocks}
We fix the vectors $$\xi_1= e_1, \xi_2 = e_2, \xi_3= e_1+e_2, \xi_4 = e_1-e_2$$ in $\R^2$. Let us make a list of all parameters that we are going to use. They will be fixed in the order below.\vspace{0,3cm}\\
\begin{tabular}[h]{l|l}
Parameter & meaning  \\
\hline
$\eta,\delta$ & Parameters in the main proposition that will ensure convergence\\
$\kappa$ & size of the ball where the error is reduced, $R_0$ is small outside $B_\kappa$\\
$\varepsilon$ &  smoothing of $\rho$ (see Section \ref{sec: perturbations})\\
$\mu_1$ & concentration \\
$\mu_2$ & very high concentration\\
$\omega$ & phase speed\\
$\lambda$ & oscillation
\end{tabular}\hspace{0,3cm}\\
Let $\Phi:\R\rightarrow\R$ be a smooth, odd function with support in $(-\frac{1}{2},\frac{1}{2})$, and $\int\Phi\,\dx = 0$ such that $\varphi := \Phi^{'''}$ satisfies $\int\varphi^2\,\dx = 1$. Furthermore, we denote by $\varphi_\mu^k$ the translated function
\begin{align*}
\varphi_\mu^k(x) = \varphi_\mu\left(x-\frac{k}{16}|\xi_k|^2\right).
\end{align*}
The translation will ensure the disjointness of the supports of different building blocks, we will prove this in  \mbox{Lemma \ref{lemma: supports}.}
Let $\mu_2\gg\mu_1\gg 1$ and $\lambda,\omega\gg 1$ with $\lambda\in\N$ to be fixed in \mbox{Section \ref{sec: proof of main prop}.} For $k=1,2,3,4$, let us introduce
\begin{align*}
w_k(x) &= \varphi^k_{\mu_1}(\lambda x_1)\varphi_{\mu_2}(\lambda x_2),\\
w_k^c(x)&=  -\frac{\mu_1}{\mu_2}(\varphi')_{\mu_1}^k(\lambda x_1) (\Phi'')_{\mu_2}(\lambda x_2),\\
w_k^{cc}(x) &= -\frac{1}{\lambda\mu_2}\varphi_{\mu_1}^k(\lambda x_1)(\Phi'')_{\mu_2}(\lambda x_2),\\
q_k(x) &= \frac{1}{\omega}(\varphi^k_{\mu_1})^2(\lambda x_1)\varphi_{\mu_2}^2(\lambda x_2).
\end{align*}
\begin{lemma}\label{lemma: small w est}
It holds
\begin{align*}
 \int_{\tor} w_k^2 \,\dx &=1,\\
 \int_{\tor} w_k \,\dx &= \int_{\tor} w_k^c \,\dx = \int_{\tor} w_k^{cc} \,\dx = 0.
\end{align*}
For any $s\in [1,\infty]$, we have the estimates
\begin{align*}
\|\partial_1^{l_1}\partial_2^{l_2} w_k\|_{L^s(\tor)} &\leq C(s) \lambda^{l_1+l_2}\mu_1^{l_1+\frac{1}{2}-\frac{1}{s}}\mu_2^{l_2+\frac{1}{2}-\frac{1}{s}},\\
\|\partial_1^{l_1}\partial_2^{l_2} w_k^c\|_{L^s(\tor)} & \leq C(s)\lambda^{l_1+l_2}\mu_1^{l_1+\frac{3}{2}-\frac{1}{s}}\mu_2^{l_2-\frac{1}{2}-\frac{1}{s}},\\
 \|\partial_1^{l_1}\partial_2^{l_2} w_k^{cc}\|_{L^r(\tor)}&\leq C(s)\lambda^{l_1+l_2-1} \mu_1^{l_1+\frac{1}{2}-\frac{1}{s}}\mu_2^{l_2-\frac{1}{2}  - \frac{1}{s}},\\
 \|\partial_1^{l_1}\partial_2^{l_2} q_k\|_{L^s(\tor)} &\leq C(s) \omega^{-1}\lambda^{l_1+l_2}\mu_1^{l_1+1-\frac{1}{s}}\mu_2^{l_2+1-\frac{1}{s}}.
\end{align*}
\end{lemma}
\begin{proof}
We have 
\begin{align*}
\int_{\tor} w_k^2(x) \,\dx =\int_0^1 (\varphi_{\mu_1}^k)^2(\lambda x_1) \,\di\cdot \int_0^1 \varphi^2_{\mu_2}(\lambda x_2) \,\dii =1
\end{align*}
by  \eqref{eq: scaling} and since $\int \varphi^2\, \dx =1$. Similarly, one gets the zero mean values of $w_k, w_k^c$ and $w_k^{cc}$ by noting that $\int_{\mathbb{T}} \varphi \,\dx =0$ since $\varphi = \Phi^{'''}$ is a derivative. The estimates can also be proven using \eqref{eq: scaling}.
\end{proof}
For $k=1,2,3,4$ we define the linear rotations
\begin{align}\label{def: Lambda}
\Lambda_k:\R^2&\rightarrow\R^2,\nonumber\\
x&\mapsto (\xi_k\cdot x, \xi_k^\perp\cdot x).
\end{align}
Our main building block is now defined as
\begin{align*}
W^p_k(x,t) &= w_k\left(\Lambda_k\left(x-\omega t \frac{\xi_k}{|\xi_k|^2}\right)\right) \frac{\xi_k}{|\xi_k|}\\
&= w_k\left(\Lambda_k x - \omega t e_1\right)\frac{\xi_k}{|\xi_k|},
\end{align*}
i.e.
\begin{align*}
W^p_k(x,t)&:=W^p_{\xi_k,\mu_1,\mu_2,\lambda,\omega}(x,t) = \varphi^k_{\mu_1}(\lambda(\xi_k\cdot x-\omega t))\varphi_{\mu_2}(\lambda \xi_k^\perp\cdot x)\frac{\xi_k}{|\xi_k|},
\end{align*}
which means that we first rotate $w_k$ and move in time in the direction of $\xi_k$.
This vector field is not divergence free. We define the corrector $W^c_k$ by
\begin{align*}
W^c_k(x,t) &:=W^c_{\xi_k,\mu_1,\mu_2,\lambda,\omega}(x,t)=w_k^c\left(\Lambda_k\left(x-\omega t \frac{\xi_k}{|\xi_k|^2}\right)\right) \frac{\xi_k^\perp}{|\xi_k|}\\
&=-\frac{\mu_1}{\mu_2}(\varphi')^k_{\mu_1}(\lambda (\xi_k\cdot x-\omega t)) (\Phi'')_{\mu_2}(\lambda \xi_k^\perp\cdot x)\frac{\xi_k^\perp}{|\xi_k|}
\end{align*} 
and observe that $\dv(W^p_k + W^c_k) = 0$, see Proposition \ref{prop: building blocks}. We introduce further building blocks by
\begin{align*}
W^{cc,\parallel}_k(x,t) &:=W^{cc,\parallel}_{\xi_k,\mu_1,\mu_2,\lambda,\omega}(x,t)=w_k^{cc}\left(\Lambda_k\left(x-\omega t \frac{\xi_k}{|\xi_k|^2}\right)\right) \frac{\xi_k}{|\xi_k|}\\
&=-\frac{1}{\lambda\mu_2}\varphi^k_{\mu_1}(\lambda (\xi_k\cdot x-\omega t)) (\Phi'')_{\mu_2}(\lambda \xi_k^\perp\cdot x)\frac{\xi_k}{|\xi_k|},\\
W^{cc,\perp}_k(x,t) &:=W^{cc,\perp}_{\xi_k,\mu_1,\mu_2,\lambda,\omega}(x,t)=w_k^{cc}\left(\Lambda_k\left(x-\omega t \frac{\xi_k}{|\xi_k|^2}\right)\right) \frac{\xi_k^\perp}{|\xi_k|}\\
&=-\frac{1}{\lambda\mu_2}\varphi^k_{\mu_1}(\lambda (\xi_k\cdot x-\omega t)) (\Phi'')_{\mu_2}(\lambda \xi_k^\perp\cdot x)\frac{\xi_k^\perp}{|\xi_k|}.
\end{align*} 
Finally, we introduce the building blocks for our time-corrector
\begin{align*}
Y_k(x,t)&:=Y_{\xi_k,\mu_1,\mu_2,\lambda,\omega}(x,t)=q_k\left(\Lambda_k\left(x-\omega t \frac{\xi_k}{|\xi_k|^2}\right)\right)\xi_k\\
 &=\frac{1}{\omega}(\varphi^k_{\mu_1})^2(\lambda(\xi_k\cdot x - \omega t))(\varphi_{\mu_2})^2(\lambda\xi_k^\perp\cdot x) \xi_k
\end{align*}
We note that our building blocks are again periodic functions on $\R^2$ with period $1$ in both variables since $\xi_k\in\N^2$.
\begin{prop}[Building blocks]\label{prop: building blocks}
The building blocks are $\lambda$-periodic and  satisfy 
\begin{enumerate}[(i)]
\item $\dv( W_k^p\otimes W_k^p) = \partial_t Y_k$,
\item $\int_{\tor} W_k^p\otimes  W_k^p (x,t)\,\dx = \frac{\xi_k}{|\xi_k|}\otimes\frac{\xi_k}{|\xi_k|}$,
\item $\|W_k^p(\cdot,t)\|_{L^s(\tor)} = \|w_k\|_{L^s(\tor)}$ for all $s\in [1,\infty]$,
\item $\int_{\tor} W_k^p(x,t)\,\dx = \int_{\tor} W_k^c(x,t)\,\dx = \int_{\tor} W_k^{cc,\parallel}(x,t)\,\dx = \int_{\tor} W_k^{cc,\perp}(x,t)\,\dx = 0$.
\end{enumerate}
Furthermore, for all $k\in\N, l\in\N$ they satisfy the following estimates:
\begin{align*}
\|\nabla^l W_k^p\|_{L^s([-k,k]^2)} &\leq C(s) k^\frac{2}{s}\lambda^l\mu_1^{\frac{1}{2}-\frac{1}{s}}\mu_2^{l+\frac{1}{2}-\frac{1}{s}},\\
\|\nabla^l W_k^c\|_{L^s([-k,k]^2)}  &\leq C(s) k^\frac{2}{s}\lambda^l\mu_1^{\frac{3}{2}-\frac{1}{s}}\mu_2^{l-\frac{1}{2}-\frac{1}{s}},\\
\|\nabla^l W_k^{cc,\parallel}\|_{L^s([-k,k]^2)}  &\leq   C(s)k^\frac{2}{s} \lambda^{l-1} \mu_1^{\frac{1}{2}-\frac{1}{s}}\mu_2^{l-\frac{1}{2}  - \frac{1}{s}},\\
\|\nabla^l W_k^{cc,\perp}\|_{L^s([-k,k]^2)}  &\leq  C(s)k^\frac{2}{s} \lambda^{l-1} \mu_1^{\frac{1}{2}-\frac{1}{s}}\mu_2^{l-\frac{1}{2}  - \frac{1}{s}},\\
\|\nabla^l Y_k\|_{L^s([-k,k]^2)}  &\leq  C(s)k^\frac{2}{s} \omega^{-1}\lambda^l\mu_1^{1-\frac{1}{s}}\mu_2^{l+1-\frac{1}{s}}
\end{align*}
\end{prop}
\begin{proof}
For $(i)$, we have by Lemma \ref{lemma: div of matrix} with $f(x) = (\varphi^k_{\mu_1})^2(\lambda (x - \omega t))$ and $g(x)=\varphi_{\mu_2}^2(\lambda x)$
\begin{align*}
\dv( W_k^p\otimes W_k^p) & = \dv \left((\varphi^k_{\mu_1})^2(\lambda (\xi_k\cdot x - \omega t))\varphi_{\mu_2}^2(\lambda \xi_k^\perp\cdot x)\frac{\xi_k}{|\xi_k|}\otimes\frac{\xi_k}{|\xi_k|} \right)\\
&= \lambda \left((\varphi^k_{\mu_1})^2\right)'(\lambda (\xi_k\cdot x - \omega t))\varphi_{\mu_2}^2(\lambda\xi_k^\perp\cdot x)\xi_k\\
&=\partial_t Y_k.
\end{align*}
For $(ii)$, this is immediate for $k=1$, since by Lemma \ref{lemma: small w est}
\begin{align*}
\int_{\tor} W_k^p\otimes  W_k^p \,\dx &= \int_{\tor} w_k^2(x-\omega t e_1) \,\dx\cdot \frac{\xi_k}{|\xi_k|}\otimes\frac{\xi_k}{|\xi_k|} = \frac{\xi_k}{|\xi_k|}\otimes\frac{\xi_k}{|\xi_k|}.
\end{align*}
The same is true for $k=2$ by switching the roles of $x_1$ and $x_2$ in the definition of $w_k$. For $k=3$, we calculate with the transformation rule by rotating the cube $[-\frac{1}{2},\frac{1}{2}]^2$ by $\Lambda_k$
\begin{align*}
\int_{\tor} W_k^p\otimes  W_k^p \,\dx  &= \int_{[-\frac{1}{2},\frac{1}{2}]^2} w_k^2(\Lambda_k x- \omega t e_1)\,\dx\cdot \frac{\xi_k}{|\xi_k|}\otimes\frac{\xi_k}{|\xi_k|}\\
&= \frac{1}{|\det D\Lambda_k|}\int_Q w_k^2(x-\omega t e_1) \,\dx \cdot \frac{\xi_k}{|\xi_k|}\otimes\frac{\xi_k}{|\xi_k|} 
\end{align*}
where $Q = \Lambda_k([-\frac{1}{2},\frac{1}{2}]^2)$ is the by 90 degress rotated and scaled cube with vertices $\left\lbrace \pm e_1, \pm e_2\right\rbrace$. It is not difficult to see that, by a geometric argument,  it holds $\int_Q w_k^2 \,\dx= 2\int_{\tor} w_k^2 \,\dx$ because $w_k$ is periodic. Since $|\det D\Lambda_k|=2$ for $k=3$, we have
\begin{align*}
\int_{\tor} W_k^p\otimes  W_k^p \,\dx = \int_{\tor} w_k^2(x-\omega t e_1) \,\dx \cdot \frac{\xi_k}{|\xi_k|}\otimes\frac{\xi_k}{|\xi_k|} = \frac{\xi_k}{|\xi_k|}\otimes\frac{\xi_k}{|\xi_k|},
\end{align*}
and the same reasoning holds for $k=4$. For $(iii)$, we do a similar calculation and obtain
\begin{align*}
\|W_k(\cdot,t)\|_{L^s(\tor)}^s=\int_{\tor} |W_k^p|^s \,\dx = \int_{\tor} |w_k(x-\omega t e_1)|^s\,\dx = \|w_k\|^s_{L^s(\tor)}
\end{align*}
for any $s\in [1,\infty)$, and the same calculations show $(iv)$. The estimates follow directly from \mbox{Lemma \ref{lemma: small w est}} and exploiting the fact that $\mu_2\gg\mu_1$.
\end{proof}
\begin{lemma}[Disjointness of supports]\label{lemma: supports}
We have 
\begin{align*}
\supp W_k^p = \supp W_k^c =\supp W_k^{cc,\parallel}=\supp W_k^{cc,\perp} = \supp Y_k
\end{align*}
and for large enough $\mu_1$ (independent of $\lambda, \mu_2$) it holds
\begin{align*}
\supp W_{k_1}^p \cap \supp W_{k_2}^p = \emptyset
\end{align*}
for $k_1\neq k_2$.
\end{lemma}
\begin{proof}
Looking at the definition, we see that the function $w_k$ (and also $w_k^c, w_k^{cc}, q_k$) is supported in small balls of radius $\frac{1}{\lambda\mu_1}$ around the points $\frac{1}{\lambda}\left((\frac{1}{2},\frac{1}{2})+\frac{k}{16}|\xi_k|^2e_1+\Z^2\right),$ i.e.
\begin{align*}
\supp w_k\subset B_{\frac{1}{\lambda\mu_1}}(0) +\frac{1}{\lambda}\left(\left(\frac{1}{2},\frac{1}{2}\right)+\frac{k}{16}|\xi_k|^2e_1+\Z^2\right).
\end{align*}
Therefore, for a fixed time $t$, we have since $W^p_k(x,t) = w_k\left(\Lambda_k\left(x-\omega t \frac{\xi_k}{|\xi_k|^2}\right)\right) \frac{\xi_k}{|\xi_k|}$
\begin{align*}  
\supp W_k^p(\cdot,t)\subset B_{\frac{1}{\lambda\mu_1}}(0)  +\frac{1}{\lambda}\Lambda_k^{-1}\left(\left(\frac{1}{2},\frac{1}{2}\right)+\frac{k}{16}|\xi_k|^2e_1+\Z^2\right) + \omega t\frac{\xi_k}{|\xi_k|^2},
\end{align*}
i.e. we calculate, using $\Lambda_k^{-1} = \frac{1}{2}\Lambda_k$,
\begin{align*}
\supp W_1^p(\cdot, t)&\subset  B_{\frac{1}{\lambda\mu_1}}(0) + \frac{1}{\lambda}\left(\frac{1}{2}, \frac{1}{2}\right)+ \frac{1}{\lambda}\frac{1}{16} \xi_1 + \frac{1}{\lambda}\Z^2 + \omega t (1,0),\\
\supp W_2^p(\cdot, t)&\subset  B_{\frac{1}{\lambda\mu_1}}(0) +  \frac{1}{\lambda}\left(\frac{1}{2}, \frac{1}{2}\right) + \frac{1}{\lambda}\frac{1}{8} \xi_2 +\frac{1}{\lambda}\Z^2 +  \omega t(0,1),\\
\supp W_3^p(\cdot, t)&\subset  B_{\frac{1}{\lambda\mu_1}}(0) + \frac{1}{\lambda} \left(\frac{1}{2}, 0\right) + \frac{1}{\lambda}\frac{3}{16} \xi_3 + \frac{1}{\lambda}\left(\frac{1}{2}\Z\right)^2 + \omega t \left(\frac{1}{2},\frac{1}{2}\right),\\
\supp W_4^p(\cdot, t)&\subset  B_{\frac{1}{\lambda\mu_1}}(0) + \frac{1}{\lambda} \left(0,-\frac{1}{2}\right) + \frac{1}{\lambda} \frac{1}{4}\xi_4 +  \frac{1}{\lambda}\left(\frac{1}{2}\Z\right)^2 + \omega t \left(\frac{1}{2},-\frac{1}{2}\right).
\end{align*}
One can now check by hand that the supports are disjoint. We do this for $W_2^p$ and $W_4^p$ as an example. Assume there is an $x\in \supp W_2^p(\cdot,t) \cap W_4^p(\cdot,t)$. Then there exists $y_1,y_2\in  B_{\frac{1}{\lambda\mu_1}}(0)$ and $k\in\Z^2$, $l\in (\frac{1}{2}\Z)^2$ such that
\begin{align*}
y_1 + \frac{1}{\lambda}\left(\frac{1}{2},\frac{1}{2}\right) + \frac{1}{\lambda}\frac{1}{8}\xi_2 + \frac{1}{\lambda} k + \omega t (0,1) = x =y_2 + \frac{1}{\lambda}\left(0,-\frac{1}{2}\right) + \frac{1}{\lambda}\frac{1}{4}\xi_4 + \frac{1}{\lambda} l + \omega t \left(\frac{1}{2},-\frac{1}{2}\right)
\end{align*}
or equivalently
\begin{align*}
\underbrace{y_1-y_2}_{\in  B_{\frac{2}{\lambda\mu_1}}(0)}  &= -\frac{1}{\lambda} \left(\frac{1}{2},1\right) + \lambda\left(\frac{2}{8},-\frac{3}{8}\right) + \frac{1}{\lambda}(l-k) + \omega t \left(\frac{1}{2},-\frac{3}{2}\right)\\
&=\underbrace{ -\frac{1}{\lambda} \left(\frac{1}{2},1\right)  + \frac{1}{\lambda}(l-k) }_{\in \frac{1}{\lambda}(\frac{1}{2}\Z)^2} +\frac{1}{\lambda}\left(\frac{1}{8},0\right) + \underbrace{\frac{1}{\lambda}\left(\frac{1}{8},-\frac{3}{8}\right) + \omega t \left(\frac{1}{2},-\frac{3}{2}\right)}_{\in \lbrace s (1,-3): s\in \R\rbrace}.
\end{align*}
But it is not difficult to see that  $0\notin \frac{1}{\lambda}(\frac{1}{2}\Z)^2 +  \frac{1}{\lambda}\left(\frac{1}{8},0\right) + \left\lbrace s(1,-3):s\in\R\right\rbrace$. Therefore, we can choose $\mu_1$ large enough such that $B_{\frac{2}{\lambda\mu_1}}(0)\cap \left(\frac{1}{\lambda}(\frac{1}{2}\Z)^2 +  \frac{1}{\lambda}\left(\frac{1}{8},0\right) + \left\lbrace s(1,-3):s\in\R\right\rbrace\right)=\emptyset$. This shows $\supp W_2^p(\cdot,t)\cap \supp W_4^p(\cdot,t) = \emptyset$.
\end{proof}

\section{The perturbations}\label{sec: perturbations}
Before we can define the perturbations, let us decompose the error $\Rtr$ in the following way. There are smooth functions $\Gamma_k$ with $|\Gamma_k|\leq 1$ such that for any matrix $A$ with $|A-I|<\frac{1}{8}$
\begin{align}\label{eq: decomposition}
A=\sum_k\Gamma_k^2(A)\frac{\xi_k}{|\xi_k|}\otimes\frac{\xi_k}{|\xi_k|},
\end{align}
see Section 5 in \cite{brue2021nonuniqueness}.
Let $\kappa\in\N$ such that
\begin{align}\label{wishlist: kappa1}
\|\Rtr(t)\|_{L^1(\R^2\setminus B_\kappa)}\leq \frac{\eta}{2}
\end{align}
for all $t\in[0,1]$. With condition \eqref{wishlist: kappa1}, our choice of $\kappa$ is set.  For $\varepsilon>0$ we further define
\begin{align*}
\gamma(t)&= \frac{e(t)(1-\frac{\delta}{2})-\int_{\R^2}|u_0|^2(x,t)\,\dx}{2\|\chi_\kappa\|^2_{L^2(\R^2)}},\\
\rho(x,t)&=10\sqrt{\varepsilon^2+|\Rtr(x,t)|^2} + \gamma(t),\\
a_k(x,t)&=\chi_\kappa(x)\rho^\frac{1}{2}(x,t)\Gamma_k\left(I+\frac{\Rtr(x,t)}{\rho(x,t)}\right),
\end{align*}
noting that the decomposition \eqref{eq: decomposition} exists for $I+\frac{\Rtr}{\rho}$. The function $\chi_\kappa$ is a smooth cutoff with $\chi_\kappa\equiv 1$ on $B_\kappa$ and $\chi_\kappa\equiv 0$ on $\R^2\setminus B_{\kappa+1}$. For later use, we note that
\begin{align}\label{eq: R coeff}
\chi_\kappa^2(x)\rho(x,t) I + \chi_\kappa^2(x) \Rtr(x,t) = \sum_k a_k^2(x,t)\frac{\xi_k}{|\xi_k|}\otimes\frac{\xi_k}{|\xi_k|}.
\end{align}
We define
\begin{align*}
H^k(x,t) = \frac{a_k(x,t)}{|\xi_k|} w_k^{cc}\left(\Lambda_k \left(x- \omega t\frac{\xi_k}{|\xi_k|^2}\right)\right).
\end{align*}
Let us define the perturbations as follows. 
\begin{align*}
w(x,t) &= \sumk \nabla^\perp H^k(x,t),\\
u^t(x,t)&= -\sumk\mathbb{P}\left(a_k^2(x,t)  Y_k(x,t)\right),\\
v(x,t)&=\mathbb{P}\int_0^t r_0(x,s)\,\ds.
\end{align*}
We note that 
\begin{align}\label{residual divergence}
\dv w= 0,
\end{align}
being an orthogonal gradient.
We set $$u_{1} = u_0 +w+ u^t + v.$$
By a simple calculation, we see that 
\begin{align*}
\nabla^\perp H^k (x,t)& = a_k(x,t) W_k^p(x,t) + a_k(x,t) W_k^c(x,t) + w_k^{cc}\left(\Lambda_k\left( x-\omega t\frac{\xi_k}{|\xi_k|^2}\right)\right) \frac{\nabla^\perp a_k(x,t)}{|\xi_k|}\\
&=a_k(x,t) W_k^p(x,t) + a_k(x,t) W_k^c(x,t)\\
&\hspace{0,3cm} + \frac{\langle\nabla^\perp a_k(x,t)\cdot\xi_k\rangle}{|\xi_k|^2} W_k^{cc,\parallel} (x,t)  + \frac{\langle\nabla^\perp a_k(x,t)\cdot\xi^\perp_k\rangle}{|\xi_k|^2} W_k^{cc,\perp} (x,t)
\end{align*}
and we set $w= u^p + u^c$ with
\begin{align}
u^p(x,t)&=\sumk  a_k(x,t) W_k^p(x,t),\nonumber\\
u^c(x,t) &= \sumk  a_k(x,t) W_k^c(x,t) +b_k^1(x,t) W_k^{cc,\parallel} (x,t)  + b_k^2(x,t) W_k^{cc,\perp} (x,t)\label{decomposition of uc}
\end{align}
where we denote
\begin{align*}
b_k^1(x,t) &= \frac{\langle\nabla^\perp a_k(x,t)\cdot\xi_k\rangle}{|\xi_k|^2},\\
b_k^2(x,t) &= \frac{\langle\nabla^\perp a_k(x,t)\cdot\xi^\perp_k\rangle}{|\xi_k|^2}.
\end{align*}
\begin{rem}
We note several things:
\begin{enumerate}
\item $a_k(x,t)$: Decomposition of the old error $R_0$, also pumping energy into the system.
\item $\mathbb{P}$ denotes the Leray projector.
\item Note that $\dv( w + u^t+v) = 0$ and thus also $\dv u_1 = 0$.
\item We will sometimes use $w=\sum_k \nabla^\perp  H^k$ as a whole and use estimates on $H^k$, whereas on other occasions we have to decompose $w= u^p + u^c$ and use certain properties of the individual parts.
\end{enumerate}
\end{rem}

\begin{lemma}
The function $u_1$ is smooth with $u_1\in C([0,1],L^2(\R^2)\cap L^3(\R^2))$, i.e. $u_1$ has the desired regularity.
\end{lemma}
\begin{proof}
The function $u_0$ is smooth and in $ C([0,1],L^2(\R^2)\cap L^3(\R^2))$ by assumption. For $w$  this is also clear since it is smooth with compact support. For $u^t$, we note that $\mathbb{P}:L^s(\R^2)\rightarrow L^s(\R^2)$ is a bounded operator for all $1<s<\infty$. Since the function inside $\mathbb{P}$ in the definition of $u^t$ is smooth and compactly supported and therefore in $C([0,1],L^2(\R^2)\cap L^3(\R^2))$, this also holds for $u^t.$ By assumption, $r_0$ is smooth and $r_0\in C([0,1],L^\infty(\R^2))$ with compact support in space, in particular also $r_0\in C([0,1],L^2(\R^2)\cap L^3(\R^2))$ and therefore also $v\in  C([0,1],L^2(\R^2)\cap L^3(\R^2))$ by the boundedness of $\mathbb{P}$.
\end{proof}

\section{Estimates of the perturbations}\label{sec: estimates perturbations}\label{sec. estimates of perturbations}
In this section, we provide the necessary estimates on the perturbations. We start with a preliminary estimate on the coefficients $a_k$ and then estimate the individual parts of the perturbations separately. After that, we obtain an estimate on the energy increment and conclude the section by fixing the parameter $\varepsilon$.
\begin{lemma}[Preliminary estimates I]
It holds
\begin{align}
\|a_k\|_{C^l(\R^2\times [0,1])}&\leq C(R_0, u_0, e,\delta,\kappa,\varepsilon,l),\label{est ak}
\end{align}
and
\begin{align}\label{eq: a_k in L2}
\|a_k(\cdot,t)\|_{L^2(\R^2)}\leq \sqrt{10\pi}\left((\kappa+1)\varepsilon^\frac{1}{2} + \delta^\frac{1}{2}\right)
\end{align}
uniformly in $t$.
\end{lemma}
\begin{proof}
For the first part, we only note that by \eqref{eq: assumption energy}
\begin{align*}
0\leq e(t)\left(1-\frac{\delta}{2}\right) - \int_{\R^2} |u_0|^2(x,t)\,\dx  \leq \frac{3}{4}\delta,
\end{align*}
so we have $0\leq\gamma(t)\leq\frac{3}{4}\delta/(2|B_\kappa|)$. This together with the definition of $a_k$ implies the $L^\infty$-estimates.  For the second part, we calculate using $|\Gamma_k|\leq 1$
\begin{align*}
\int_{\R^2} a_k^2(x,t) \,\dx &= \int_{\R^2}\chi_\kappa^2(x)\rho(x,t) \Gamma_k^2\left(I+\frac{\Rtr(x,t)}{\rho(x,t)}\right)\,\dx \leq \int_{B_{\kappa+1}} 10\sqrt{\varepsilon^2+|\Rtr(t)|^2(x,t)}+\gamma(t)  \,\dx\\
& \leq 10\pi(\kappa+1)^2\varepsilon + 20\|R_0\|_{C_tL^1_x} + 10 \pi (\kappa+1)^2 \gamma\\
&\leq 10\pi(\kappa+1)^2\varepsilon + 20\|R_0\|_{C_tL^1_x} + 5\pi \frac{(\kappa+1)^2}{\|\chi_\kappa\|^2_{L^2(\R^2)}}\left(e(t)(1-\frac{\delta}{2}) - \int_{\R^2} |u_0|^2(x,t)\,\dx\right).
\end{align*}
Using again \eqref{eq: assumption energy} and 
$$5\pi\frac{(\kappa+1)^2}{\|\chi_\kappa\|^2_{L^2(\R^2)}}\leq  5\frac{(\kappa+1)^2}{\kappa^2}\leq 20,$$ 
and also that $20\|R_0\|_{C_tL^1_x}\leq \delta$ by assumption,
we obtain
\begin{align}\label{eq: a_k in L2 quadr}
\int_{\R^2} a_k^2(x,t) \,\dx \leq 10\pi(\kappa+1)^2\varepsilon   + 16\delta.
\end{align}
From this \eqref{eq: a_k in L2} follows.
\end{proof}

\begin{lemma}[Estimate of the principal perturbation]\label{lemma: estimate up}
It holds
\begin{align*}
\|u^p(t)\|_{L^s(\R^2)}&\leq C(R_0,u_0,e,\delta, \kappa,\varepsilon)\mu_1^{\frac{1}{2}-\frac{1}{s}}\mu_2^{\frac{1}{2}-\frac{1}{s}}
\end{align*}
and for $p=2$  more refined
\begin{align}
\|u^p(t)\|_{L^2(\R^2)}&\leq \sqrt{10\pi}\left((\kappa+1)\varepsilon^\frac{1}{2} + \delta^\frac{1}{2}\right) +\frac{ C(R_0,u_0,e,\delta,\kappa,\varepsilon)}{\lambda^\frac{1}{2}}\label{wishlist vareps 2}
\end{align}
uniformly in $t$.
\end{lemma}
\begin{proof}
For the first estimate, we use Proposition \ref{prop: building blocks}, \eqref{est ak} and the fact that $u^p$ is supported in $B_{\kappa+1}$. For the second estimate, we use Proposition \ref{prop: improved hoelder}, noting again that $\supp a_k(\cdot,t)\subset [-\kappa-1,\kappa+1]^2$, Lemma \ref{lemma: small w est}, Proposition \ref{prop: building blocks} and \eqref{eq: a_k in L2}
\begin{align*}
\|\sumk a_k(\cdot,t) W^p_k(\cdot,t)\|_{L^2(\R^2)} &=\sumk\|a_k(\cdot,t) W^p_k(\cdot,t)\|_{L^2([-\kappa-1,\kappa+1]^2)}\\
&\leq \|a_k(\cdot,t)\|_{L^2(\R^2)} \|W^p_k(\cdot,t)\|_{L^2(\tor)}\\
&\hspace{0,3cm} + C \frac{2\kappa+2}{\lambda^\frac{1}{2}}\|a_k(\cdot,t)\|_{C^1(\R^2)}\|W^p_k(\cdot,t)\|_{L^2(\tor)} \\
&\leq \sqrt{10\pi}\left((\kappa+1)\varepsilon^\frac{1}{2} + \delta^\frac{1}{2}\right) + \frac{C(R_0,u_0, e,\kappa,\delta,\kappa,\varepsilon)}{\lambda^\frac{1}{2}}.\qedhere
\end{align*}.
\end{proof}

\begin{lemma}[Estimates of the correctors]\label{lemma: estimate correctors}
We have
\begin{align*}
\|u^c(t)\|_{L^s(\R^2)}&\leq C(R_0,u_0,e,\delta,\kappa,\varepsilon) \mu_1^{\frac{3}{2}-\frac{1}{s}}\mu_2^{-\frac{1}{2}-\frac{1}{s}}
\end{align*}
and
\begin{align*}
\|u^t(t)\|_{L^2(\R^2)} &\leq C(R_0,u_0,e,\delta,\kappa,\varepsilon) \frac{\mu_1^\frac{1}{2}\mu_2^\frac{1}{2}}{\omega}
\end{align*}
uniformly in $t$.
\end{lemma}
\begin{proof}
This proof follows by using Proposition \ref{prop: building blocks} together with \eqref{est ak} and the fact that $u^c$ is supported in $[-\kappa-1,\kappa+1]^2$. For $u^t$, we also use that $\mathbb{P}$ is bounded from $L^2$ to $L^2$ and the argument inside $\mathbb{P}$ in the definition of $u^t$ is supported in $[-\kappa-1,\kappa+1]^2$.
\end{proof}

\begin{lemma}\label{lemma: estimate H}
It holds
\begin{align*}
\|\nabla^l H^k(t)\|_{L^s(\tor)} \leq C(R_0, u_0, e,\delta,\kappa,\varepsilon)\lambda^{l-1} \mu_{1}^{\frac{1}{2}-\frac{1}{s}}\mu_2^{l-\frac{1}{2}-\frac{1}{s}}
\end{align*}
uniformly in $t$.
\end{lemma}

\begin{proof}
This follows immediately from Lemma \ref{lemma: small w est}, \eqref{est ak} and the definition of $H^k$, exploiting also the fact that $\mu_2\gg\mu_1$.
\end{proof}

\begin{lemma}\label{lemma: v in L2}
It holds for all $s\in (1,\infty)$
\begin{align*}
\|v(t)\|_{L^s(\R^2)}\leq \|\mathbb{P}\|_{\mathcal{L}(L^s(\R^2))}\|r_0\|_{C_tL^s_x}
\end{align*}
and for $s=2$
\begin{align*}
\|v(t)\|_{L^2(\R^2)}\leq \|r_0\|_{C_tL^2_x}
\end{align*}
uniformly in $t$.
\end{lemma}
\begin{proof}
This follows using Minkowski's inequality and the fact that $\mathbb{P}: L^s(\R^2)\rightarrow L^s(\R^2)$ for all $s\in (1,\infty)$. For $s=2$, $\mathbb{P}$ is an orthogonal projection, therefore $\|\mathbb{P}\|_{\mathcal{L}(L^2(\R^2))}\leq 1$.
\end{proof}

\begin{lemma}[Estimate of the energy increment]\label{lemma: energy inc}
We have 
\begin{align}
\left|e(t) \left(1-\frac{\delta}{2}\right) - \int_{\R^2} |u_1|^2(x,t)\,\dx\right|&\leq \frac{1}{32}\delta  +20 \pi (\kappa +1)^2 \varepsilon \nonumber\\
&\hspace{0,3cm} +  C(R_0,r_0, u_0,e, \delta,\kappa,\varepsilon)\left(\mu_1^{-\frac{1}{6}}\mu_2^{-\frac{1}{6}} + \frac{\mu_1^{\frac{1}{2}}\mu_2^{\frac{1}{2}}}{\omega} +\frac{\mu_1}{\mu_2} + \frac{1}{\lambda} \right).\label{wishlist eps sigma}
\end{align}
\end{lemma}

\begin{proof}
Looking at \eqref{eq: R coeff}, we consider
\begin{align*}
u^p\otimes u^p - \chi^2_\kappa \Rtr &= \chi^2_\kappa\rho I+\sumk  a_k^2 \left(W_k^p\otimes W_k^p - \frac{\xi_k}{|\xi_k|}\otimes\frac{\xi_k}{|\xi_k|}\right).
\end{align*}
We take the trace and use that $\Rtr$ is traceless, hence we get
\begin{align*}
|u^p|^2 - 2\chi^2_\kappa\gamma(t) &= 20 \chi^2_\kappa\sqrt{\varepsilon^2 + |\Rtr|^2} + \sumk a_k^2 \left(|W_k^p|^2 - 1\right)
\end{align*}
Integrating this and using $\sqrt{\varepsilon^2 + |x|^2}\leq \varepsilon + |x|$, we get
\begin{align}
\left|\int_{\R^2} |u^p|^2(x,t) \,\dx - \left(e(t)\left(1-\frac{\delta}{2}\right)-\int_{\R^2}|u_0|^2(x,t)\,\dx\right)\right| &\leq 20 \pi (\kappa +1)^2 \varepsilon + 40 \|R_0\|_{C_tL^1_x}\nonumber\\
&+ \sum_k \left|\int_{\R^2}  a_k^2(x,t)\left(|W_k^p|^2(x,t) -1\right)\,\dx\right|.\label{eq: energy inc up}
\end{align}
We can estimate each summand in the second line with Lemma \ref{lemma: fast oscillations}, using that $a_k$ is supported in $[-\kappa-1,\kappa+1]^2$, \eqref{est ak} and Proposition \ref{prop: building blocks} by
\begin{align}
\left|\int_{\R^2}  a_k^2(x,t) \left(|W_k^p|^2(x,t) -1\right)\,\dx\right|&\leq \frac{4\sqrt{2}(\kappa+1)^2\|a_k^2(\cdot,t)\|_{C^1(\R^2)} \||W_k^p(\cdot,t)|^2-1\|_{L^1(\tor)}}{\lambda}\nonumber\\
&\leq \frac{C(R_0,u_0,e,\delta,\kappa,\varepsilon)}{\lambda} .\label{eq: oscillation wk}
\end{align}
Writing $u_1 = u_0 + u^p+u^c+ u^t+ v$,  we have 
\begin{align*}
\int_{\R^2} |u_1|^2(x,t)\,\dx &= \|u_0(t)\|_{L^2(\R^2)}^2 + \|u^p(t)\|_{L^2(\R^2)}^2 + \|v(t)\|_{L^2(\R^2)}^2 + 2 \int_{\R^2} u_0\cdot v (x,t) \,\dx\\
&\hspace{0,3cm} +2 \int_{\R^2} u_0\cdot (u^p + u^c + u^t)(x,t) \,\dx + 2\int_{\R^2} u^p \cdot(u^c + u^t +v)(x,t)\,\dx\\
&\hspace{0,3cm} + 2\int_{\R^2} v\cdot(u^c + u^t)(x,t)\,\dx + \int_{\R^2}|u^c+u^t|^2(x,t)\,\dx,
\end{align*}
where by Lemma \ref{lemma: estimate up}, Lemma \ref{lemma: estimate correctors} and Lemma \ref{lemma: v in L2} 
\begin{align}
2\left| \int_{\R^2} u_0 \cdot(u^p + u^c + u^t)(x,t) \,\dx \right|& \leq 2\|u_0(t)\|_{L^2(\R^2)} \left(\|u^c(t)\|_{L^2(\R^2)}  +\|u^t(t)\|_{L^2(\R^2)}\right)\nonumber\\
&\hspace{0,3cm} + 2\|u_0(t)\|_{L^3(\R^2)}\|u^p(t)\|_{L^\frac{3}{2}(\R^2)}\nonumber\\
&\leq C(R_0,u_0, e,\delta,\kappa,\varepsilon) \left(\frac{\mu_1}{\mu_2}+\mu_1^{-\frac{1}{6}}\mu_2^{-\frac{1}{6}} +\frac{\mu_1^{\frac{1}{2}}\mu_2^{\frac{1}{2}}}{\omega} \right),\nonumber\\
2\left|\int_{\R^2} u^p\cdot(u^c + u^t +v)(x,t)\,\dx\right|&\leq 2\|u^p(t)\|_{L^2(\R^2)} \left(\|u^c(t)\|_{L^2(\R^2)} +  \|u^t(t)\|_{L^2(\R^2)}\right)\nonumber\\
&\hspace{0,3cm}+ 2 \|\mathbb{P}\|_{\mathcal{L}(L^3(\R^2))}  \|u^p(t)\|_{L^\frac{3}{2}(\R^2)} \|r_0\|_{C_tL^3_x},\nonumber\\
&\leq C(R_0,r_0,u_0,e,\delta,\kappa,\varepsilon) \left(\frac{\mu_1}{\mu_2}+\mu_1^{-\frac{1}{6}}\mu_2^{-\frac{1}{6}} + \frac{\mu_1^{\frac{1}{2}}\mu_2^{\frac{1}{2}}}{\omega} \right),\label{est: rlin2}\\
 2\left|\int_{\R^2} v\cdot(u^c  + u^t)(x,t)\,\dx\right| &\leq 2\|v(t)\|_{L^2(\R^2)}\left(\|u^c(t)\|_{L^2(\R^2)}  + \|u^t(t)\|_{L^2(\R^2)}\right)\nonumber\\
&\leq  C(R_0,r_0,u_0,e,\delta,\kappa,\varepsilon) \left(\frac{\mu_1}{\mu_2} + \frac{\mu_1^{\frac{1}{2}}\mu_2^{\frac{1}{2}}}{\omega}\right),\nonumber\\
\int_{\R^2}|u^c+u^t|^2(x,t)\,\dx &\leq 2\left(\|u^c(t)\|^2_{L^2(\R^2)} + \|u^t(t)\|^2_{L^2(\R^2)}\right)\nonumber\\
&\leq C(R_0,u_0, e,\delta,\kappa,\varepsilon)\left(\left(\frac{\mu_1}{\mu_2}\right)^2  + \left(\frac{\mu_1^{\frac{1}{2}}\mu_2^{\frac{1}{2}}}{\omega}\right)^2\right).\nonumber
\end{align}
This yields
\begin{align*}
\left|e(t) \left(1-\frac{\delta}{2}\right) - \int_{\R^2} |u_1|^2(x,t)\,\dx\right| &\leq \left|\int_{\R^2} |u^p|^2(x,t) \,\dx - \left(e(t)\left(1-\frac{\delta}{2}\right)-\int_{\R^2}|u_0|^2(x,t)\,\dx\right)\right|\\
&\hspace{0,3cm} + \|v(t)\|_{L^2(\R^2)}^2 + 2\|u_0(t)\|_{L^2(\R^2)}\|v(t)\|_{L^2(\R^2)}\\
&\hspace{0,3cm} +C(R_0,r_0, u_0,e, \delta,\kappa,\varepsilon)\left(\mu_1^{-\frac{1}{6}}\mu_2^{-\frac{1}{6}} + \frac{\mu_1^{\frac{1}{2}}\mu_2^{\frac{1}{2}}}{\omega} +\frac{\mu_1}{\mu_2} \right)\\
&\leq  \left|\int_{\R^2} |u^p|^2(x,t) \,\dx - \left(e(t)(1-\frac{\delta}{2})-\int_{\R^2}|u_0|^2(x,t)\,\dx\right)\right|\\
&\hspace{0,3cm} + \|r_0\|_{C_tL^2_x}^2 + 2\|u_0(t)\|_{L^2(\R^2)}\|r_0\|_{C_tL^2_x}\\
&\hspace{0,3cm} + C(R_0,r_0, e,u_0, \delta,\kappa,\varepsilon)\left(\mu_1^{-\frac{1}{6}}\mu_2^{-\frac{1}{6}} + \frac{\mu_1^{\frac{1}{2}}\mu_2^{\frac{1}{2}}}{\omega} +\frac{\mu_1}{\mu_2} \right).
\end{align*}
Let us combine the previous inequality with \eqref{eq: energy inc up} and \eqref{eq: oscillation wk} and then use our assumptions \eqref{eq: assumption r} and $e(t)\geq \frac{1}{2}$, this yields
\begin{align*}
\left|e(t) \left(1-\frac{\delta}{2}\right) - \int_{\R^2} |u_1|^2(x,t)\,\dx\right| &\leq   40 \|R_0\|_{C_tL^1_x} + \|r_0\|_{C_tL^2_x}^2 + 2\|u_0(t)\|_{L^2(\R^2)}\|r_0\|_{C_tL^2_x}\\
&\hspace{0,3cm}+ 20 \pi (\kappa +1)^2 \varepsilon \\
&\hspace{0,3cm}  + C(R_0,r_0, u_0, e,\delta,\kappa,\varepsilon)\left(\mu_1^{-\frac{1}{6}}\mu_2^{-\frac{1}{6}} + \frac{\mu_1^{\frac{1}{2}}\mu_2^{\frac{1}{2}}}{\omega} +\frac{\mu_1}{\mu_2} +\frac{1}{\lambda}\right)\\
&\leq \frac{1}{32}\delta  +20 \pi (\kappa +1)^2 \varepsilon  \\
&\hspace{0,3cm} +  C(R_0,r_0, u_0, e,\delta,\kappa,\varepsilon)\left(\mu_1^{-\frac{1}{6}}\mu_2^{-\frac{1}{6}} + \frac{\mu_1^{\frac{1}{2}}\mu_2^{\frac{1}{2}}}{\omega} +\frac{\mu_1}{\mu_2} + \frac{1}{\lambda} \right).\qedhere
\end{align*}
\end{proof}
At this point, we fix $\varepsilon$ and choose this parameter so small such that
\begin{align*}
20 \pi (\kappa +1)^2 \varepsilon &<\frac{1}{32}\delta,\\
\sqrt{10}\pi ((\kappa+1)\varepsilon^\frac{1}{2}+\delta^{\frac{1}{2}})&\leq 10\delta^{\frac{1}{2}},
\end{align*}
therefore \eqref{wishlist vareps 2} becomes
\begin{align*}
\|u^p(t)\|_{L^2(\R^2)}&\leq 10 \delta^\frac{1}{2} +\frac{ C(R_0,u_0,e,\delta,\kappa,\varepsilon)}{\lambda^\frac{1}{2}}
\end{align*}
and \eqref{wishlist eps sigma} reduces to
\begin{align}\label{energy increment}
\left|e(t) \left(1-\frac{\delta}{2}\right) - \int_{\R^2} |u_1|^2(x,t)\,\dx\right|&< \frac{1}{16}\delta  +  C(R_0,r_0, u_0,e, \delta,\kappa,\varepsilon)\left(\mu_1^{-\frac{1}{6}}\mu_2^{-\frac{1}{6}} + \frac{\mu_1^{\frac{1}{2}}\mu_2^{\frac{1}{2}}}{\omega} +\frac{\mu_1}{\mu_2} + \frac{1}{\lambda} \right)\nonumber\\
&\leq \frac{1}{8}\delta e(t) + C(R_0,r_0, u_0, e,\delta,\kappa,\varepsilon)\left(\mu_1^{-\frac{1}{6}}\mu_2^{-\frac{1}{6}} + \frac{\mu_1^{\frac{1}{2}}\mu_2^{\frac{1}{2}}}{\omega} +\frac{\mu_1}{\mu_2} + \frac{1}{\lambda} \right).
\end{align}

\section{Estimates of the curl in Hardy space}\label{sec: curl estimates}
In the following Lemmas, we prove that the curls of the perturbations are in the real Hardy space $H^p(\R^2)$ for $\frac{2}{3}<p<1$ and estimate their Hardy space seminorms in terms of $\lambda, \mu_1$ and $\mu_2$. We will use  Remark \ref{rem: hardy atoms}; therefore, we decompose the perturbations into finitely many functions that are supported on disjoint, very small balls of radius $\frac{1}{\lambda\mu_1}$.
\begin{lemma}[Curl of $w$]\label{lemma: curl w}
It holds $\curl w(t)\in H^p(\R^2)$ and
\begin{align*}
\|\curl w(t)\|_{H^p(\R^2)}\leq C(R_0,u_0,e,\delta,\kappa,\varepsilon) \lambda\mu_1^{\frac{1}{2}-\frac{2}{p}}\mu_2^\frac{3}{2}\text{ for all } t\in[0,1].
\end{align*}
\end{lemma}

\begin{proof}
By definition of $H^k$, $\supp H^k = \supp W_k^p.$ As seen in the proof of Lemma \ref{lemma: supports}, for a fixed time $t$, the perturbations are supported in small, disjoint balls of radius $\frac{1}{\lambda\mu_1}$ around the points in the  finite set $$M_k(t)= \left\lbrace\frac{1}{\lambda}\Lambda_k^{-1}\left(\left(\frac{1}{2},\frac{1}{2}\right)+\frac{k}{16}|\xi_k|^2e_1+\Z^2\right) + \omega t\frac{\xi_k}{|\xi_k|^2}, k=1,2,3,4\right\rbrace\cap B_{\kappa+1}.$$ Let us abbreviate $B_{x_0} = B_{\frac{1}{\lambda\mu_1}}(x_0)$ for $x_0\in M(t)$, and let us decompose $w$ as
\begin{align*}
w(x,t)= \sum_{x_0\in M(t)} \theta_{x_0}(x,t)
\end{align*}
where
\begin{align*}
\theta_{x_0}(x,t) =\mathbb{1}_{B(x_0)}(x)w(x,t).
\end{align*}
Since $\theta_{x_0}$ is smooth and has compact support, $\curl\theta_{x_0}\in H^p(\R^2)$ since, as a derivative of a compactly supported function, it satisfies  $\int_{\R^2}\curl\theta_{x_0} \,\dx = 0$. We estimate the $H^p$-seminorm for each $\curl\theta_{x_0}.$ We have 
\begin{align*}
\curl \theta_{x_0}(x,t) &= \mathbb{1}_{B(x_0)}(x)\curl  w(x,t) = -\mathbb{1}_{B(x_0)}(x)\sumk\Delta H^k(x,t).
\end{align*}
 As already said, each $\theta_{x_0}$ is supported on one ball of measure $\frac{C}{\lambda\mu_1}$.  By Lemma \ref{lemma: estimate H},
\begin{align*}
\|\curl\theta_{x_0}(t)\|_{L^\infty(\R^2)}\leq \sumk\|\nabla^2 H^k(t)\|_{L^\infty(\R^2)}\leq C(R_0,u_0,e,\delta,\kappa,\varepsilon)\lambda \mu_{1}^{\frac{1}{2}}\mu_2^{\frac{3}{2}}
\end{align*}
This gives us by Remark \ref{rem: hardy atoms}
\begin{align*}
\|\curl\theta_{x_0}(t)\|_{H^p(\R^2)}\leq C(R_0,u_0,e,\delta,\kappa,\varepsilon) \lambda^{1-\frac{2}{p}}\mu_1^{\frac{1}{2}-\frac{2}{p}}\mu_2^{\frac{3}{2}}.
\end{align*}
Since $|M(t)|$ is of order $\kappa^2\lambda^2$, $\curl w$ is made up of $\approx\lambda^2\kappa^2$- many functions $\curl\theta_{x_0}$, and we obtain
\begin{align*}
\|\curl w(t)\|^p_{H^p(\R^2)}&\leq \sum_{x_0\in M(t)}\|\curl\theta_{x_0}(t)\|^p_{H^p(\R^2)} \leq C(R_0,u_0,e,\delta,\kappa,\varepsilon) \lambda^2 \lambda^{p-2} \mu_1^{\frac{p}{2}-2}\mu_2^\frac{3p}{2}\\ &=  C(R_0,u_0,e,\delta,\kappa,\varepsilon) \lambda^{p} \mu_1^{\frac{p}{2}-2}\mu_2^\frac{3p}{2} .\qedhere
\end{align*}
\end{proof}

\begin{lemma}[Curl of $u^t$]\label{lemma: curl ut}
It holds $\curl u^t(t)\in H^p(\R^2)$ and
\begin{align*}
\|\curl u^{t}(t)\|_{H^p(\R^2)}\leq C(R_0,u_0,e,\delta,\kappa,\varepsilon) \omega^{-1}\lambda\mu_1^{1-\frac{2}{p}}\mu_2^2 \text{ for all } t\in[0,1].
\end{align*}
\end{lemma}

\begin{proof}
We write again 
\begin{align*}
u^t(x,t) = \sum_{x_0\in M(t)} \theta_{x_0}(x,t)
\end{align*}
with the same decomposition as in the previous Lemma.
Since $\curl (\mathbb{P} f) = \curl f$ for all smooth $f:\mathbb{R}^2\rightarrow\R^2$, we have
\begin{align*}
\curl \theta_{x_0}(x,t)& =\mathbb{1}_{B(x_0)} \sumk a_k^2(x,t)  \frac{\lambda}{\omega}(\varphi_{\mu_1}^k)^2(\lambda(\xi_k\cdot x - \omega t))(\varphi_{\mu_2}^2)'(\lambda\xi_k^\perp\cdot x)|\xi_k|^2\\
&= \mathbb{1}_{B(x_0)}\sumk a_k^2(x,t) (\partial_2 q_k)(\Lambda_k x- \omega t e_1).
\end{align*}
Arguing in the same way as before, we just need to estimate with Lemma \ref{lemma: small w est}
\begin{align*}
\|\curl\theta_{x_0}(t)\|_{L^\infty(\R^2)}\leq C(R_0,u_0,e,\delta,\kappa,\varepsilon)\omega^{-1}\lambda\mu_1\mu_2^{2},
\end{align*}
hence
\begin{equation*}
\|\curl u^t(t)\|_{H^p(\R^2)}\leq C(R_0,u_0,e,\delta,\kappa,\varepsilon) \omega^{-1}\lambda\mu_1^{1-\frac{2}{p}}\mu_2^2.\qedhere
\end{equation*}
\end{proof}

\begin{lemma}[Curl of $v$]\label{lemma: curl v}
It holds 
\begin{align*}
\|\curl v(t)\|_{H^p(\R^2)} = \|\int_0^t \curl r_0(s)\,\ds\|_{H^p(\R^2)}.\qedhere
\end{align*}
\end{lemma}

\begin{proof}
This is true since  $\curl (\mathbb{P} f) = \curl f$ for all smooth $f:\mathbb{R}^2\rightarrow\R^2$, which gives
\begin{equation*}
\curl v(t) = \int_0^t \curl r_0(s) \,\ds.\qedhere
\end{equation*}
\end{proof}

\section{The new error}
This section is devoted to the definition of the new error $(r_1,R_1)$, which will be estimated in the next section.
\subsection{The new Reynolds-defect-equation}
Plugging  $u_1$ into the new Reynolds-defect-equation and writing $u_1 = u_0 + w+u^t+v,$ we need to define $(r_1,R_1, p_1)$ such that
\begin{align}
-r_1 - \dv \overset{\circ}{R_1} &= \dv(u_0\otimes (u_1 - u_0) + (u_1-u_0)\otimes u_0)\nonumber\\
&\hspace{0,3cm} + \dv((u_1-u_0 - u^p)\otimes u^p) + \dv(u^p\otimes (u_1-u_0-u^p))\nonumber\\
&\hspace{0,3cm}  + \dv((u_1-u_0-u^p)\otimes (u_1-u_0-u^p))\nonumber\\
&\hspace{0,3cm} + \partial_t u^t + \dv(u^p\otimes u^p - \Rtr)\nonumber\\
&\hspace{0,3cm} + \partial_t (u^p+ u^c )\nonumber\\
&\hspace{0,3cm} + \partial_t v - r_0\nonumber\\
&\hspace{0,3cm} + \nabla(p_1 - p_0).\label{eq: iteration eq}
\end{align}
We will analyse each line in \eqref{eq: iteration eq} in separate subsections.
\subsection{\texorpdfstring{Analysis of the first three lines of \eqref{eq: iteration eq}}{Analysis of the first three lines of the iteration equation}}
Let us define
\begin{align*}
R^{\operatorname{lin},1} &= u_0\otimes (u_1 - u_0) + (u_1-u_0)\otimes u_0,\\
R^{\operatorname{lin},2}& = (u_1-u_0 - u^p)\otimes u^p + u^p\otimes (u_1-u_0-u^p),\\
R^{\operatorname{lin},3}&=  (u_1-u_0-u^p)\otimes (u_1-u_0-u^p),
\end{align*}
i.e.
\begin{align*}
&\dv(u_0\otimes (u_1 - u_0) + (u_1-u_0)\otimes u_0)\\
+ &\dv((u_1-u_0 - u^p)\otimes u^p) + \dv(u^p\otimes (u_1-u_0-u^p))\\
+ &\dv((u_1-u_0-u^p)\otimes (u_1-u_0-u^p))\\
&= \dv(R^{\operatorname{lin},1} + R^{\operatorname{lin},2} + R^{\operatorname{lin},3}).
\end{align*}

\subsection{\texorpdfstring{Analysis of the fourth line of \eqref{eq: iteration eq}}{Analysis of the fourth line of the iteration equation}}
\subsubsection{\texorpdfstring{Rewriting the fourth line of \eqref{eq: iteration eq}}{Rewriting the fourth line of the iteration equation}}
Using that $$u^t(x,t)= -\sumk\mathbb{P}\left(a_k^2(x,t)  Y_k(x,t)\right) = -\sumk a_k^2(x,t)  Y_k(x,t) - \nabla p^t$$ for some $p^t$ and \eqref{eq: R coeff}, let us start by calculating
\begin{align*}
\partial_t u^t + \dv(u^p\otimes u^p - \Rtr)&= -\sumk \partial_t a_k^2 Y_k -\sumk a_k^2\partial_t Y_k\\
&\hspace{0,3cm} + \dv\left(\sumk a_k^2 W_k^p\otimes W_k^p\right)\\
&\hspace{0,3cm} + \dv\left(-\sum_k a_k^2 \frac{\xi_k}{|\xi_k|}\otimes \frac{\xi_k}{|\xi_k|}\right)\\
&\hspace{0,3cm} + \dv(\chi_\kappa^2\Rtr - \Rtr) - \nabla (\partial_t p^t) + \nabla(\chi_\kappa^2\rho)
\end{align*}
We consider the second and third summand  on the right hand side of the previous calculation. By Proposition \ref{prop: building blocks}, we have
\begin{align*}
-\sumk a_k^2 \partial_t Y_k+\dv\left(\sumk a_k^2 W_k^p\otimes W_k^p\right) &= \sumk a_k^2  \underbrace{\left(\dv\left(W_k^p\otimes W_k^p\right) - \partial_t Y_k\right)}_{=0}\\
&\hspace{0,3cm} + \sumk  (W_k^p\otimes W_k^p)\cdot\nabla a_k^2
\end{align*}
Also, we have
\begin{align*}
\dv\left(-\sum_k a_k^2 \frac{\xi_k}{|\xi_k|}\otimes \frac{\xi_k}{|\xi_k|}\right) &= -\sumk\left(\frac{\xi_k}{|\xi_k|}\otimes \frac{\xi_k}{|\xi_k|}\right)\cdot\nabla a_k^2.
\end{align*}
Putting together the previous two calculations, the fourth line in \eqref{eq: iteration eq} equals
\begin{align*}
\partial_t u^t + \dv(u^p\otimes u^p - \overset{\circ}{R_0})&= -\sumk\partial_t a_k^2  Y_k+ \sumk    \left(W_k^p\otimes W_k^p - \frac{\xi_k}{|\xi_k|}\otimes \frac{\xi_k}{|\xi_k|}\right)\cdot \nabla a_k^2 \\
&\hspace{0,3cm} + \dv(\chi_\kappa^2\Rtr - \Rtr) - \nabla (\partial_t p^t) + \nabla(\chi_\kappa^2\rho)\\
& =r^Y + \dv R^Y + r^{\operatorname{quad}}+\dv R^{\operatorname{quad}}  +\dv R^{\kappa} - \nabla \pi_1,
\end{align*}
where we can directly define
\begin{align*}
R^{\kappa} &= \chi_\kappa^2 \Rtr - \Rtr,\\
\pi_1 &= \partial_t p^t - \chi_\kappa^2\rho.
\end{align*}

\subsubsection{\texorpdfstring{Definition of  $R^{\operatorname{quad}}$ and $r^{\operatorname{quad}}$}{Definition of the quadratic error} }
We define $R^{\quadr}$, $r^{\quadr}$ as
\begin{align*}
(r^{\quadr},R^{\quadr}) = \sumk\tilde{S}_N\left(\nabla a_k^2, W_k^p\otimes W_k^p - \frac{\xi_k}{|\xi_k|}\otimes \frac{\xi_k}{|\xi_k|}\right),
\end{align*}
with an $N\in\N$ to be chosen in Section \ref{sec: proof of main prop}. Hence, by construction
\begin{align*}
r^{\quadr} + \dv R^{\quadr} = \sumk    \left(W_k^p\otimes W_k^p - \frac{\xi_k}{|\xi_k|}\otimes \frac{\xi_k}{|\xi_k|}\right)\cdot \nabla a_k^2.
\end{align*}

\subsubsection{\texorpdfstring{Definition of $R^Y$ and $r^Y$}{Definition of the error from the time corrector}}
We add and subtract
\begin{align*}
-\sumk  \partial_t a_k^2(x,t) Y_k(x,t) &= -\sumk   \partial_t a_k^2(x,t)\left( Y_k(x,t) - \frac{1}{\omega}\xi_k\right) - \sumk  \frac{1}{\omega}\partial_t a_k^2(x,t) \xi_k.
\end{align*}
Noting that $\int_{\tor} Y_k\,\dx = \frac{1}{\omega}\xi_k,$ we can define
\begin{align*}
(r^{Y,1},R^Y) =-\sumk S_N(\partial_t a_k^2, Y_k-\frac{1}{\omega}\xi_k)
\end{align*}
so that by definition
\begin{align*}
r^{Y,1} + \dv R^Y =  -\sumk   \partial_t a_k^2\left( Y_k - \frac{1}{\omega}\xi_k\right).
\end{align*}
We further define
\begin{align*}
r^{Y,2} = - \sumk  \frac{1}{\omega} \partial_t a_k^2(x,t)\xi_k
\end{align*}
and set
\begin{align*}\
r^Y= r^{Y, 1} + r^{Y, 2},
\end{align*}
hence
\begin{align*}
r^Y + \dv R^Y = -\sumk  \partial_t a_k^2 Y_k.
\end{align*}

\subsection{\texorpdfstring{Analysis of the fifth line of \eqref{eq: iteration eq}}{Analysis of the fifth line of the iteration equation}}
We will write the third line in the form 
\begin{align*}
\partial_t(u^p+u^c) = r^{\tme} + \dv R^{\tme}.
\end{align*}
We will use the operators from Definition \ref{def: antidivergence by hand}.
Calculating, we see that
\begin{align*}
\partial_t u^p(x,t) &= \sumk \partial_t a_k(x,t) W_k^p(x,t) + \sumk a_k(x,t) \partial_t W_k^p(x,t)\\
&= \sumk\partial_t a_k(x,t) W_k^p(x,t) + \omega\lambda\mu_1\sumk a_k(x,t) (\varphi')^k_{\mu_1}(\lambda(\xi_k\cdot x-\omega t))\varphi_{\mu_2}(\lambda\xi_k^\perp\cdot x)\frac{\xi_k}{|\xi_k|}\\
&=  \left( r^{\tme,1} + \dv R^{\tme,1}\right) + \dv\left(\frac{\omega\lambda\mu_1}{|\xi_k|}\sumk a_k A((\varphi')_{\mu_1}^k (\lambda(\cdot - \omega t)), \varphi_{\mu_2}(\lambda\cdot),\xi_k)\right)\\
&\hspace{0,3cm}  - \frac{\omega\lambda\mu_1}{|\xi_k|}\sumk A((\varphi')_{\mu_1}^k (\lambda(\cdot - \omega t)), \varphi_{\mu_2}(\lambda\cdot),\xi_k)\cdot \nabla a_k\\
&=( r^{\tme,1} + \dv R^{\tme,1}) + \dv\tilde{R}^{\tme,2} +(r^{\tme,2}+ \dv R^{\tme,2}).
\end{align*} 
with 
\begin{align*}
(r^{\tme,1},R^{\tme,1}) &= \sumk S_N(\partial_t a_k, W_k^p),\\
\tilde{R}^{\tme,2} &= \frac{\omega\lambda\mu_1}{|\xi_k|}\sumk a_k A((\varphi')_{\mu_1}^k (\lambda(\cdot - \omega t)), \varphi_{\mu_2}(\lambda\cdot),\xi_k),\\
( r^{\tme,2},R^{\tme,2}) &= -\frac{\omega\lambda\mu_1}{|\xi_k|}\sumk\tilde{S}_N(\nabla a_k,A((\varphi')_{\mu_1}^k (\lambda(\cdot - \omega t)), \varphi_{\mu_2}(\lambda\cdot),\xi_k)).
\end{align*}
Analogously, let us write, using \eqref{decomposition of uc}
\begin{align}\label{partial uc}
\partial_t u^c(x,t) &=\sumk  \partial_ta_k(x,t) W_k^c(x,t) +a_k(x,t) \partial_t W_k^c(x,t)\nonumber\\
&\hspace{0,3cm} + \sumk \partial_t b_k^1(x,t) W_k^{cc,\parallel}(x,t) + b_k^1(x,t)\partial_t W_k^{cc,\parallel}(x,t)\nonumber)\\
&\hspace{0,3cm} +\sumk \partial_t b_k^2(x,t) W_k^{cc,\perp}(x,t) + b_k^2(x,t)\partial_t W_k^{cc,\perp}(x,t) .
\end{align}
The first line  of \eqref{partial uc} can be written as
\begin{align*}
&\sumk  \partial_ta_k(x,t) W_k^c(x,t) +\sumk a_k(x,t) \partial_t W_k^c(x,t)\\
&=\sumk  \partial_ta_k(x,t) W_k^c(x,t) + \frac{\omega\lambda\mu_1^2}{\mu_2}\sumk a_k(x,t) (\varphi'')^k_{\mu_1}(\lambda(\xi_k\cdot x-\omega t))(\Phi'')_{\mu_2}(\lambda\xi_k^\perp\cdot x)\frac{\xi_k}{|\xi_k|}\\
&= \left( r^{\tme,3} + \dv R^{\tme,3}\right) +\dv\left(\frac{\omega\lambda\mu_1^2}{\mu_2|\xi_k|} \sumk a_k B((\varphi'')^k_{\mu_1}(\lambda(\cdot -\omega t)),(\Phi'')_{\mu_2}(\lambda\cdot),\xi_k)\right)\\
&\hspace{0,3cm} - \frac{\omega\lambda\mu_1^2}{\mu_2|\xi_k|} \sumk  B((\varphi'')^k_{\mu_1}(\lambda(\cdot -\omega t)),(\Phi'')_{\mu_2}(\lambda\cdot),\xi_k)\cdot\nabla a_k\\
&= \left( r^{\tme,3} + \dv R^{\tme,3}\right) + \dv\tilde{R}^{\tme,4} +( r^{\tme,4} + \dv R^{\tme,4})
\end{align*}
with
\begin{align*}
(r^{\tme,3},R^{\tme,3}) &=\sumk S_N(\partial_ta_k, W_k^c),\\
\tilde{R}^{\tme,4} & = \frac{\omega\lambda\mu_1^2}{\mu_2|\xi_k|} \sumk a_k B((\varphi'')^k_{\mu_1}(\lambda(\cdot -\omega t)),(\Phi'')_{\mu_2}(\lambda\cdot),\xi_k),\\
 (r^{\tme,4} , R^{\tme,4}) & = -\frac{\omega\lambda\mu_1^2}{\mu_2|\xi_k|} \sumk\tilde{S}_N\left(\nabla a_k,B((\varphi'')^k_{\mu_1}(\lambda(\cdot -\omega t)),(\Phi'')_{\mu_2}(\lambda\cdot),\xi_k)\right).
\end{align*}
For the second line of \eqref{partial uc}, we write
\begin{align*}
 &\sumk \partial_t b_k^1(x,t) W_k^{cc,\parallel}(x,t) + \sumk b_k^1(x,t)\partial_t W_k^{cc,\parallel}(x,t)\\ 
&= \sumk \partial_t b_k^1(x,t) W_k^{cc,\parallel}(x,t) + \frac{\omega\mu_1}{\mu_2}\sumk  b_k^1(x,t) (\varphi')^k_{\mu_1}(\lambda(\xi_k\cdot x-\omega t))(\Phi'')_{\mu_2}(\lambda\xi_k^\perp\cdot x)\frac{\xi_k}{|\xi_k|}\\
 &=(r^{\tme,5} + \dv R^{\tme,5}) + \dv\left(\frac{\omega\mu_1}{\mu_2|\xi_k|}\sumk b_k^1 A((\varphi')^k_{\mu_1}(\lambda(\cdot -\omega t)),(\Phi'')_{\mu_2}(\lambda\cdot),\xi_k)\right)\\
 &\hspace{0,3cm} - \frac{\omega\mu_1}{\mu_2|\xi_k|}\sumk  A((\varphi')^k_{\mu_1}(\lambda(\cdot -\omega t)),(\Phi'')_{\mu_2}(\lambda\cdot),\xi_k)\cdot\nabla b_k^1\\
 &=(r^{\tme,5} + \dv R^{\tme,5}) + \dv\tilde{R}^{\tme,6} +( r^{\tme,6}+\dv R^{\tme,6})
\end{align*}
with
\begin{align*}
(r^{\tme,5} ,\dv R^{\tme,5}) &= \sumk S_N(\partial_t b_k^1,W_k^{cc,\parallel}),\\
\tilde{R}^{\tme,6} &= \frac{\omega\mu_1}{\mu_2|\xi_k|}\sumk b_k^1 A((\varphi')^k_{\mu_1}(\lambda(\cdot -\omega t)),(\Phi'')_{\mu_2}(\lambda\cdot),\xi_k),\\
(r^{\tme,6} ,\dv R^{\tme,6}) &= -\frac{\omega\mu_1}{\mu_2|\xi_k|}\sumk\tilde{S}_N\left(\nabla b_k^1, A((\varphi')^k_{\mu_1}(\lambda(\cdot -\omega t)),(\Phi'')_{\mu_2}(\lambda\cdot),\xi_k)\right).
\end{align*}
Similarly, we have for the third line of \eqref{partial uc}
\begin{align*}
 &\sumk \partial_t b_k^2(x,t) W_k^{cc,\perp}(x,t) + \sumk b_k^2(x,t)\partial_t W_k^{cc,\perp}(x,t)\\
&=  \sumk \partial_t b_k^2(x,t) W_k^{cc,\perp}(x,t) +  \frac{\omega\mu_1}{\mu_2}\sumk b_k^2(x,t)(\varphi')^k_{\mu_1}(\lambda(\xi_k\cdot x - \omega t))(\Phi'')_{\mu_2}(\lambda\xi_k^\perp\cdot x)\frac{\xi_k^\perp}{|\xi_k|}\\
  &=(r^{\tme,7} + \dv R^{\tme,7}) +\dv\left( \frac{\omega\mu_1}{\mu_2|\xi_k|}\sumk b_k^2 B((\varphi')^k_{\mu_1}(\lambda(\cdot -\omega t)),(\Phi'')_{\mu_2}(\lambda\cdot),\xi_k)\right)\\
 &\hspace{0,3cm} - \frac{\omega\mu_1}{\mu_2|\xi_k|}\sumk  B((\varphi')^k_{\mu_1}(\lambda(\cdot -\omega t)),(\Phi'')_{\mu_2}(\lambda\cdot),\xi_k)\cdot\nabla b_k^2\\
 &= (r^{\tme,7} + \dv R^{\tme,7} )+ \dv\tilde{R}^{\tme,8} + (r^{\tme,8}+\dv R^{\tme,8})
\end{align*}
with
\begin{align*}
(r^{\tme,7} ,\dv R^{\tme,7}) &= \sumk S_N(\partial_t b_k^2,W_k^{cc,\perp}),\\
\tilde{R}^{\tme,8} &= \frac{\omega\mu_1}{\mu_2|\xi_k|}\sumk b_k^2 B((\varphi')^k_{\mu_1}(\lambda(\cdot -\omega t)),(\Phi'')_{\mu_2}(\lambda\cdot),\xi_k),\\
(r^{\tme,8} ,\dv R^{\tme,8}) &= -\frac{\omega\mu_1}{\mu_2|\xi_k|}\sumk\tilde{S}_N\left(\nabla b_k^2, B((\varphi')^k_{\mu_1}(\lambda(\cdot -\omega t)),(\Phi'')_{\mu_2}(\lambda\cdot),\xi_k)\right).
\end{align*}

Finally, we set
\begin{align*}
R^{\tme} &= \sum_{i=1}^8 R^{\tme,i}+\sum_{i=1}^4\tilde{R}^{\tme,2i}\\
r^{\tme} &= \sum_{i=1}^8 r^{\tme,i}.
\end{align*}

\subsection{\texorpdfstring{Analysis of the sixth and seventh line of \eqref{eq: iteration eq}}{Analysis of the sixth and seventh line of the iteration equation}}
Since $\mathbb{P}r_0 = r_0 - \nabla p^r$ for some $p^r$, we see that
\begin{align*}
\partial_t v - r_0 & = -\nabla p^r,
\end{align*}
i.e. it only remains a part that can be put into the new pressure and we define
\begin{align*}
\pi_2 = p^r.
\end{align*}
\subsection{Definition of the new error}
Altogether, we define
\begin{align*}
R^1 &= -\left(R^{\operatorname{lin},1} + R^{\operatorname{lin},2} +R^{\operatorname{lin},3} + R^\kappa +R^{\quadr}+ R^Y  + R^{\tme} \right),\\
r_1 &= -\left(r^{\quadr}+ r^Y + r^{\tme}\right),\\
p_1 &= p_0 + \pi_1 + \pi_2 +\frac{1}{2}\operatorname{tr}R^1.
\end{align*}

\section{Estimates of the new error}
We will now estimate the different parts of $R_1$ and $r_1$ that were defined in the previous section.
\subsection{\texorpdfstring{Estimates of the symmetric tensor $R_1$}{Estimates of the matrix error}}
\begin{lemma}[Estimate of $R^{\operatorname{lin},1} $]\label{lemma: est rlin}
It holds
\begin{align*}
\|R^{\operatorname{lin},1}(t) \|_{L^1(\R^2)}\leq C(R_0,u_0,e,\delta,\kappa,\varepsilon)\left(\mu_1^{-\frac{1}{6}}\mu_2^{-\frac{1}{6}} + \frac{\mu_1^{\frac{1}{2}}\mu_2^{\frac{1}{2}}}{\omega}\right) +2 \|r_0\|_{C_tL^2_x} \|u_0(t)\|_{L^2(\R^2)}.
\end{align*}
\end{lemma}

\begin{proof}
Using Hölder's inequality and Lemma \ref{lemma: estimate up}, Lemma \ref{lemma: estimate correctors} and Lemma \ref{lemma: v in L2}, we have
\begin{align*}
\|R^{\operatorname{lin},1}(t)\|_{L^1(\R^2)} &\leq 2 \|u_0(t)\|_{L^2(\R^2)}\left(\|u^c(t)\|_{L^2(\R^2)} + \|u^t(t)\|_{L^2(\R^2)}+ \|v(t)\|_{L^2(\R^2)} \right)\\
&\hspace{0,3cm}   + 2  \|u_0(t)\|_{L^3(\R^2)}\|u^p(t)\|_{L^\frac{3}{2}(\R^2)} \\
&\leq C(R_0,u_0,e,\delta,\kappa,\varepsilon)\left(\frac{\mu_1}{\mu_2}+\mu_1^{-\frac{1}{6}}\mu_2^{-\frac{1}{6}} + \frac{\mu_1^{\frac{1}{2}}\mu_2^{\frac{1}{2}}}{\omega}\right) +2 \|r_0\|_{C_tL^2_x} \|u_0(t)\|_{L^2(\R^2)}.\qedhere
\end{align*}
\end{proof}

\begin{lemma}[Estimate of $R^{\operatorname{lin},2} $]\label{lemma: est rlin2}
It holds
\begin{align*}
\|R^{\operatorname{lin},2}(t) \|_{L^1(\R^2)}\leq  C(R_0,r_0,e,\delta,\kappa,\varepsilon) \left(\frac{\mu_1}{\mu_2}+\mu_1^{-\frac{1}{6}}\mu_2^{-\frac{1}{6}} + \frac{\mu_1^{\frac{1}{2}}\mu_2^{\frac{1}{2}}}{\omega}\right).
\end{align*}
\end{lemma}

\begin{proof}
We have
\begin{align*}
\|R^{\operatorname{lin},2}(t) \|_{L^1(\R^2)}&\leq 2 \|u^p(t)\|_{L^2(\R^2)}\left( \|u^c(t)\|_{L^2(\R^2)}  + \|u^t(t)\|_{L^2(\R^2)}\right)\\
&\hspace{0,3cm} + 2 \|u^p(t)\|_{L^\frac{3}{2}(\R^2)} \|v(t)\|_{L^3(\R^2)}
\end{align*}
and this was already estimated in \eqref{est: rlin2}.
\end{proof}

\begin{lemma}[Estimate of $R^{\operatorname{lin},3} $]\label{lemma: est rlin3}
It holds
\begin{align*}
\|R^{\operatorname{lin},3}(t) \|_{L^1(\R^2)}\leq  C(R_0,u_0,e,\delta,\kappa,\varepsilon)\left(\left(\frac{\mu_1}{\mu_2}\right)^2 + \left(\frac{\mu_1^\frac{1}{2}\mu_2^\frac{1}{2}}{\omega}\right)^2\right) + 4\|r_0\|_{C_tL^2_x}^2.
\end{align*}
\end{lemma}

\begin{proof}
Since $$R^{\operatorname{lin},3}=  (u_1-u_0-u^p)\otimes (u_1-u_0-u^p) = (u^c + u^t + v)\otimes (u^c+ u^t + v),$$ we have
\begin{align*}
\|R^{\operatorname{lin},3}(t) \|_{L^1(\R^2)}\leq 4\left(\|u^c(t)\|_{L^2(\R^2)}^2  + \|u^t(t)\|^2_{L^2(\R^2)} + \|v(t)\|^2_{L^2(\R^2)}\right),
\end{align*}
hence the assertion follows from Lemma \ref{lemma: estimate correctors} and Lemma \ref{lemma: v in L2}.
\end{proof}

\begin{lemma}[Estimate of $R^\kappa$]\label{lemma: est rkappa}
It holds 
\begin{align*}
\|R^\kappa(t)\|_{L^1(\R^2)}\leq \frac{\eta}{2}.
\end{align*}
\end{lemma}

\begin{proof}
This holds because of our choice of $\kappa$ in \eqref{wishlist: kappa1}.
\end{proof}

\begin{lemma}[Estimate of $R^{\quadr}$]\label{lemma: est rquadr}
It holds 
\begin{align*}
\|R^{\quadr}(t)\|_{L^1(\R^2)}\leq \frac{C(R_0,u_0,e,\delta,\kappa,\varepsilon,N)}{\lambda}.
\end{align*}
\end{lemma}

\begin{proof}
This follows directly from Remark \ref{rem: improved antidiv}, the scaling of $W_k^p$ (see Proposition \ref{prop: building blocks}) and the estimates on $a_k$ in \eqref{est ak}.
\end{proof}

\begin{lemma}[Estimate of $R^{Y}$]\label{lemma: est ry}
It holds 
\begin{align*}
\|R^{Y}(t)\|_{L^1(\R^2)}\leq \frac{C(R_0,u_0,e,\delta,\kappa,\varepsilon,N)}{\omega\lambda}.
\end{align*}
\end{lemma}
\begin{proof}
As for $R^{\quadr}$, this is a direct application of Remark \ref{rem: improved antidiv}.
\end{proof}

\begin{lemma}[Estimate of $R^{\tme}$]\label{lemma: est rtime}
It holds
\begin{align*}
\|R^{\tme}(t)\|_{L^1(\R^2)} \leq C(R_0,u_0,e,\delta,\kappa,\varepsilon, N) \left( \lambda^{-1}\mu_1^{-\frac{1}{2}}\mu_2^{-\frac{1}{2}} + \omega \mu_1^\frac{1}{2}\mu_2^{-\frac{3}{2}} \right).
\end{align*}
\end{lemma}

\begin{proof}
By Remark \ref{rem: improved antidiv} and the estimates for the operators $A$ and $B$ in \eqref{est A B}, Proposition \ref{prop: building blocks} and \eqref{est ak} we have
\begin{align*}
\|R^{\tme,1}\|_{L^1(\R^2)}&\leq C(\kappa)\|\partial_t a_k\|_{C^{N-1}(\R^2)}\frac{1}{\lambda}\|W_k^p\|_{L^1(\tor)}\\
&\leq  C(R_0,u_0,e,\kappa,\delta,\varepsilon,N) \lambda^{-1}\mu_1^{-\frac{1}{2}}\mu_2^{-\frac{1}{2}},\\
\|\tilde{R}^{\tme,2}\|_{L^1(\R^2)}&\leq C(\kappa)  \omega\lambda\mu_1 \|a_k\|_{C(\R^2)} \|A((\varphi')^k_{\mu_1}(\lambda(\cdot - \omega t)),\varphi_{\mu_2}(\lambda\cdot),\xi_k)\|_{L^1(\tor)}\\
&\leq  C(R_0, u_0,e,\kappa,\delta,\varepsilon,N)\omega \mu_1^\frac{1}{2}\mu_2^{-\frac{3}{2}},\\
\|R^{\tme,2}\|_{L^1(\R^2)}&\leq C(\kappa)\omega\lambda\mu_1 \|a_k\|_{C^{N}(\R^2)} \|\dv^{-1}A((\varphi')^k_{\mu_1}(\lambda(\cdot - \omega t)),\varphi_{\mu_2}(\lambda\cdot),\xi_k)\|_{L^1(\tor)}
\\&\leq  C(R_0,u_0,e,\kappa,\delta,\varepsilon,N)\omega \lambda^{-1}\mu_1^\frac{1}{2}\mu_2^{-\frac{3}{2}},\\
\|R^{\tme,3}\|_{L^1(\R^2)}&\leq C(\kappa)\|\partial_t a_k\|_{C^{N-1}(\R^2)}\frac{1}{\lambda}\|W_k^c\|_{L^1(\tor)}\\
&\leq  C(R_0,u_0,e,\kappa,\delta,\varepsilon,N)\lambda^{-1}\mu_1^\frac{1}{2}\mu_2^{-\frac{3}{2}},\\
\|\tilde{R}^{\tme,4}\|_{L^1(\R^2)}&\leq  C(\kappa)\frac{\omega\lambda\mu_1^2}{\mu_2}\|a_k\|_{C(\R^2)}\| B((\varphi'')^k_{\mu_1}(\lambda(\cdot -\omega t)),(\Phi'')_{\mu_2}(\lambda\cdot),\xi_k)\|_{L^1(\tor)}\\
&\leq  C(R_0,u_0,e,\kappa,\delta,\varepsilon,N) \omega\mu_1^{\frac{3}{2}}\mu_2^{-\frac{5}{2}},\\
\|R^{\tme,4}\|_{L^1(\R^2)}&\leq C(\kappa) \frac{\omega\lambda\mu_1^2}{\mu_2} \|a_k\|_{C^N(\R^2)} \left\|\dv^{-1} B((\varphi'')^k_{\mu_1}(\lambda(\cdot -\omega t)),(\Phi'')_{\mu_2}(\lambda\cdot),\xi_k)\right\|\\
&\leq  C(R_0,u_0,e,\kappa,\delta,\varepsilon,N)\omega \lambda^{-1}\mu_1^{\frac{3}{2}}\mu_2^{-\frac{5}{2}},
\end{align*}
\begin{align*}
\|R^{\tme,5}\|_{L^1(\R^2)}&\leq C(\kappa)\|\partial_t b_k^1\|_{C^{N-1}(\R^2)}\frac{1}{\lambda}\|W_k^{cc,\parallel}\|_{L^1(\tor)}\\
&\leq  C(R_0,u_0,e,\kappa,\delta,\varepsilon,N)\lambda^{-2}\mu_1^{-\frac{1}{2}}\mu_2^{-\frac{3}{2}},\\
\|\tilde{R}^{\tme,6}\|_{L^1(\R^2)}&\leq C(\kappa)\frac{\omega\mu_1}{\mu_2} \|b_k^1\|_{C(\R^2)} \left\|A((\varphi')^k_{\mu_1}(\lambda(\cdot - \omega t)),(\Phi'')_{\mu_2}(\lambda\cdot),\xi_k)\right\|_{L^1(\tor)}\\
&\leq  C(R_0,u_0,e,\kappa,\delta,\varepsilon,N)\omega\lambda^{-1} \mu_1^{\frac{1}{2}}\mu_2^{-\frac{5}{2}},\\
\|R^{\tme,6}\|_{L^1(\R^2)}&\leq C(\kappa) \frac{\omega\mu_1}{\mu_2} \|b_k^1\|_{C^{N}(\R^2)} \left\|\dv^{-1}A((\varphi')^k_{\mu_1}(\lambda(\cdot - \omega t)),(\Phi'')_{\mu_2}(\lambda\cdot),\xi_k)\right\|_{L^1(\tor)}\\
&\leq  C(R_0,u_0,e,\kappa,\delta,\varepsilon,N)\omega\lambda^{-2} \mu_1^{\frac{1}{2}}\mu_2^{-\frac{5}{2}},\\
\|R^{\tme,7}\|_{L^1(\R^2)}&\leq C(\kappa)\|\partial_t b_k^2\|_{C^{N-1}(\R^2)}\frac{1}{\lambda}\|W_k^{cc,\perp}\|_{L^1(\tor)}\\
&\leq  C(R_0,u_0,e,\kappa,\delta,\varepsilon,N)\lambda^{-2}\mu_1^{-\frac{1}{2}}\mu_2^{-\frac{3}{2}},\\
\|\tilde{R}^{\tme,8}\|_{L^1(\R^2)}&\leq C(\kappa)\frac{\omega\mu_1}{\mu_2} \|b_k^2\|_{C(\R^2)} \left\|B((\varphi')^k_{\mu_1}(\lambda(\cdot - \omega t)),(\Phi'')_{\mu_2}(\lambda\cdot),\xi_k)\right\|_{L^1(\tor)}\\
&\leq  C(R_0,e,\kappa,\delta,\varepsilon)\omega\lambda^{-1} \mu_1^{\frac{1}{2}}\mu_2^{-\frac{5}{2}},\\
\|R^{\tme,8}\|_{L^1(\R^2)}&\leq  C(\kappa)\frac{\omega\mu_1}{\mu_2} \|b_k^2\|_{C^{N}(\R^2)} \left\|\dv^{-1}B((\varphi')^k_{\mu_1}(\lambda(\cdot - \omega t)),(\Phi'')_{\mu_2}(\lambda\cdot),\xi_k)\right\|_{L^1(\tor)}\\
&\leq  C(R_0,u_0,e,\kappa,\delta,\varepsilon,N)\omega\lambda^{-2} \mu_1^{\frac{1}{2}}\mu_2^{-\frac{5}{2}}.\qedhere
\end{align*}
Putting those estimate together yields the claim.
\end{proof}

\subsection{\texorpdfstring{Estimates of the vector $r_1$}{Estimates of the scalar error}}
In this subsection, we estimate the new error $r_1$. Since $r_1$ also enters into the next iteration (see the definition of $v$ in Section \ref{sec: perturbations}), we need an estimate on $\int_0^t \curl r_1(x,s)\,\ds$ in $H^p(\R^2)$ as well. The operators $S_N$ and $\tilde{S}_N$ guarantee that all parts of $r_1$ have compact supports, therefore one can use Remark \ref{rem: hardy atoms}, and we control the quantity $\left\|\int_0^t \curl r_1(\cdot,s)\,\ds\right\|_{H^p(\R^2)}$ by $\|r_1\|_{L^\infty(\R^2)}.$ 
\begin{lemma}[Estimate of $r^{\quadr}$]\label{lemma: rquad1N}
The function $r^{\quadr}$ has compact support and satisfies
\begin{align*}
\|r^{\quadr}(t)\|_{L^\infty(\R^2)} &\leq C(R_0,u_0,e,\delta,\kappa,\varepsilon,N) \frac{\mu_1\mu_2}{\lambda^{N}} \\
\|\int_0^t\curl r^{\quadr}(s)\,\ds\|_{H^p(\R^2)}&\leq  C(R_0,u_0,e,\delta,\kappa,\varepsilon,N) \frac{\mu_1\mu_2^2}{\lambda^{N-1}}.\qedhere
\end{align*}
\end{lemma}

\begin{proof}
The compact support follows from the properties of $\tilde{S}_N$. By Remark \ref{rem: improved antidiv}, Proposition \ref{prop: building blocks} and \eqref{est ak}, we can estimate
\begin{align*}
\|r^{\quadr}(s)\|_{L^\infty(\R^2)}&\leq C(\kappa)\frac{1}{\lambda^{N}} \|a_k^2\|_{C^{N+1}(\R^2)} \left\|W_k^p\otimes W_k^p - \frac{\xi_k}{|\xi_k|}\otimes \frac{\xi_k}{|\xi_k|}\right\|_{L^\infty(\tor)}\\
&\leq C(R_0, u_0, e,\delta,\kappa,\varepsilon,N) \frac{\mu_1\mu_2}{\lambda^{N}} .
\end{align*}
For the curl, we use again Remark \ref{rem: improved antidiv} and obtain
\begin{align*}
\|\curl r^{\quadr}(s)\|_{L^\infty(\R^2)}&\leq C(\kappa)\|a^2\|_{C^{N+2}(\R^2)} \left\|\nabla\dv^{-N}\left(W_k^p\otimes W_k^p\right) \right\|_{L^\infty(\tor)}\\
&\leq C(R_0,u_0,e,\delta,\kappa,\varepsilon,N) \frac{\mu_1\mu_2^2}{\lambda^{N-1}} \text{ for all } s\in [0,1]
\end{align*}
and therefore also 
\begin{align*}
\|\int_0^t\curl r^{\quadr}(s)\,\ds\|_{L^\infty(\R^2)}&\leq C(R_0,u_0,e,\delta,\kappa,\varepsilon,N) \frac{\mu_1\mu_2^2}{\lambda^{N-1}}.
\end{align*}
Using that $\int_0^t\curl r^{\quadr}(s)\,\ds$ is supported in $B_{\kappa+1}$, we can apply Remark \ref{rem: hardy atoms} and obtain
\begin{align*}
\|\int_0^t\curl r^{\quadr}(s)\,\ds\|_{H^p(\R^2)}\leq  C(R_0,u_0,e,\delta,\kappa,\varepsilon,N) \frac{\mu_1\mu_2^2}{\lambda^{N-1}}.
\end{align*}
\end{proof}

\begin{lemma}[Estimate of $r^Y$]\label{lemma: rY}
The function $r^Y$ has compact support and satisfies
\begin{align*}
\|r^Y(t)\|_{L^\infty(\R^2)} &\leq C(R_0,u_0,e,\delta,\kappa,\varepsilon, N) \left(\frac{\mu_1\mu_2}{\omega\lambda^{N}}+\frac{1}{\omega}\right),  \\
\|\int_0^t\curl r^Y(s)\,\ds\|_{H^p(\R^2)} &\leq  C(R_0,u_0,e,\delta,\kappa,\varepsilon, N) \left(\frac{\mu_1\mu_2^2}{\omega\lambda^{N-1}} + \frac{1}{\omega}\right).
\end{align*}
\end{lemma}
\begin{proof}
The compact support follows from the properties of $S_N$ for $r^{Y,1}$ and the compact support of $a_k$ for $r^{Y,2}$, respectively. By Remark \ref{rem: improved antidiv}, Proposition \ref{prop: building blocks} and \eqref{est ak}, we can estimate
\begin{align*}
\|r^{Y,1}(s)\|_{L^\infty(\R^2)}&\leq C(\kappa)\frac{1}{\lambda^{N}} \|\partial_t a_k^2\|_{C^{N+1}(\R^2)} \left\|Y_k-\frac{1}{\omega}\xi_k\right\|_{L^\infty(\tor)}\\
&\leq C(R_0,u_0,e,\delta,\kappa,\varepsilon,N) \frac{\mu_1\mu_2}{\omega\lambda^{N}} .
\end{align*}
For the curl, we use again Remark \ref{rem: improved antidiv} and obtain
\begin{align*}
\|\curl r^{Y,1}(s)\|_{L^\infty(\R^2)}&\leq C(\kappa)\|\partial_t a^2\|_{C^{N+2}(\R^2)} \|\nabla\dv^{-N}Y_k\|_{L^\infty(\tor)}\\
&\leq C(R_0,u_0,e,\delta,\kappa,\varepsilon,N) \frac{\mu_1\mu_2^2}{\omega\lambda^{N-1}} \text{ for all } s\in [0,1]
\end{align*}
and therefore also 
\begin{align*}
\|\int_0^t\curl r^{Y,1}(s)\,\ds\|_{L^\infty(\R^2)}&\leq C(R_0,u_0,e,\delta,\kappa,\varepsilon,N) \frac{\mu_1\mu_2^2}{\omega\lambda^{N-1}}.
\end{align*}
Using that $\int_0^t\curl r^{\quadr}(s)\,\ds$ is supported in $B_{\kappa+1}$, we can apply Remark \ref{rem: hardy atoms} and obtain
\begin{align*}
\|\int_0^t\curl r^{Y,1}(s)\,\ds\|_{H^p(\R^2)}\leq  C(R_0,u_0,e,\delta,\kappa,\varepsilon,N) \frac{\mu_1\mu_2^2}{\omega\lambda^{N-1}}.
\end{align*}
For $r^{Y,2}$ we immediately get
\begin{align*}
\|r^{Y,2}(s)\|_{L^\infty(\R^2)}\leq C(R_0,u_0,e,\delta,\kappa,\varepsilon,N) \frac{1}{\omega}
\end{align*}
and also
\begin{align*}
\|\int_0^t\curl r^{Y,2}(s)\,\ds\|_{L^\infty(\R^2)} \leq C(R_0,u_0,e,\delta,\kappa,\varepsilon,N) \frac{1}{\omega}
\end{align*}
so that since $\supp \left(\int_0^t\curl r^{Y,2}(s)\,\ds\right)\subset B_{\kappa+1}$, we have by Remark \ref{rem: hardy atoms}
\begin{equation*}
\|\int_0^t\curl r^{Y,2}(s)\,\ds\|_{H^p\R^2)}\leq C(R_0,u_0,e,\delta,\kappa,\varepsilon,N) \frac{1}{\omega}.\qedhere
\end{equation*}
\end{proof}

\begin{lemma}[Estimate of $r^{\tme}$]\label{lemma: rtme}
The function $r^{\tme}$ has compact support and satisfies
\begin{align*}
\|r^{\tme}(t)\|_{L^\infty(\R^2)} &\leq C(R_0,u_0,e,\delta,\kappa,\varepsilon, N) \left(\frac{\mu_1^\frac{1}{2}\mu_2^\frac{1}{2}}{\lambda^N} + \frac{\omega\mu_1^\frac{3}{2}\mu_2^{-\frac{1}{2}}}{\lambda^N}\right),  \\
\|\int_0^t\curl r^{\tme}(s)\,\ds\|_{H^p(\R^2)} &\leq  C(R_0,u_0,e,\delta,\kappa,\varepsilon, N) \left(\frac{\mu_1^\frac{1}{2}\mu_2^\frac{3}{2}}{\lambda^{N-1}} + \frac{\omega\mu_1^\frac{3}{2}\mu_2^\frac{1}{2}}{\lambda^{N-1}}\right).
\end{align*}
\end{lemma}

\begin{proof}
We estimate the different parts of $r^{\tme}$ separately.
Again, by Remark \ref{rem: improved antidiv}, Proposition \ref{prop: building blocks}  and \eqref{est ak} we have
\begin{align*}
\|r^{\tme,1}(t)\|_{L^\infty(\R^2)}&\leq C(\kappa)\frac{1}{\lambda^{N}} \|\partial_t a_k\|_{C^{N}(\R^2)} \|W_k^p\|_{L^\infty(\tor)}\leq C(R_0,u_0,e,\delta,\kappa,\varepsilon,N) \frac{\mu_1^\frac{1}{2}\mu_2^\frac{1}{2}}{\lambda^{N}}
\end{align*}
and
\begin{align*}
\|\curl r^{\tme,1}(t)\|_{L^\infty(\R^2)}&\leq C(\kappa)\|\partial_t a\|_{C^{N+1}(\R^2)} \left\|\nabla\dv^{-N}W_k^p \right\|_{L^\infty(\tor)}\\
&\leq C(R_0,u_0,e,\delta,\kappa,\varepsilon,N) \frac{\mu_1^\frac{1}{2}\mu_2^\frac{3}{2}}{\lambda^{N-1}} \text{ for all } t\in [0,1]
\end{align*}
and therefore also
\begin{align*}
\left\|\int_0^t\curl r^{\tme,1}(s)\,\ds\right\|_{H^p(\R^2)}\leq C(R_0,u_0,e,\delta,\kappa,\varepsilon,N)\frac{\mu_1^\frac{1}{2}\mu_2^\frac{3}{2}}{\lambda^{N-1}}
\end{align*}
by Remark \ref{rem: hardy atoms}. For $r^{\tme,2}$, we have, using \eqref{est A B}
\begin{align*}
\|r^{\tme,2}(t)\|_{L^\infty(\R^2)}&\leq C(\kappa) \omega\lambda\mu_1 \|a_k\|_{C^{N+1}} \left\|\dv^{-N}A((\varphi')_{\mu_1}^k (\lambda(\cdot - \omega t)), \varphi_{\mu_2}(\lambda\cdot),\xi_k)\right\|_{L^\infty(\R^2)}\\
&\leq  C(R_0,u_0,e,\delta,\kappa,\varepsilon,N)  \omega\lambda\mu_1 \frac{\mu_1^\frac{1}{2}\mu_2^{-\frac{1}{2}}}{\lambda^{N+1}} = C(R_0,e,\delta,\kappa,\varepsilon)\frac{\omega\mu_1^{\frac{3}{2}}\mu_2^{-\frac{1}{2}}}{\lambda^N},\\
\|\curl r^{\tme,2}(t)\|_{L^\infty(\R^2)}&\leq C(\kappa) \omega\lambda\mu_1 \|a_k\|_{C^{N+2}} \left\|\nabla\dv^{-N}A((\varphi')_{\mu_1}^k (\lambda(\cdot - \omega t)), \varphi_{\mu_2}(\lambda\cdot),\xi_k)\right\|_{L^\infty(\R^2)}\\
&\leq C(R_0,u_0,e,\delta,\kappa,\varepsilon,N) \omega\lambda\mu_1  \frac{\mu_1^{\frac{1}{2}}\mu_2^\frac{1}{2}}{\lambda^N} =  C(R_0,e,\delta,\kappa,\varepsilon) \frac{\omega\mu_1^\frac{3}{2}\mu_2^\frac{1}{2}}{\lambda^{N-1}}
\end{align*}
and therefore also
\begin{align*}
\|\int_0^t\curl r^{\tme,2}(s)\,\ds\|_{H^p(\R^2)}\leq C(R_0,u_0,e,\delta,\kappa,\varepsilon,N) \frac{\omega\mu_1^\frac{3}{2}\mu_2^\frac{1}{2}}{\lambda^{N-1}}.
\end{align*}
In the same manner, we estimate
\begin{align*}
\|r^{\tme,3}(t)\|_{L^\infty(\R^2)} &\leq C(\kappa) \frac{1}{\lambda^N}\|\partial_t a_k\|_{C^N(\R^2)}\|W_k^c\|_{L^\infty(\R^2)}\\
&\leq C(R_0,u_0,e,\delta,\kappa,\varepsilon,N) \frac{\mu_1^\frac{3}{2}\mu_2^{-\frac{1}{2}}}{\lambda^{N}},\\
\|r^{\tme,4}(t)\|_{L^\infty(\R^2)}&\leq C(\kappa)\frac{\omega\lambda\mu_1^2}{\mu_2}\|a_k\|_{C^{N+1}(\R^2)}\left\|\dv^{-N}B((\varphi'')^k_{\mu_1}(\lambda(\cdot -\omega t)),(\Phi'')_{\mu_2}(\lambda\cdot),\xi_k)\right\|_{L^\infty(\tor)},\\
&\leq  C(R_0,u_0,e,\delta,\kappa,\varepsilon,N)\frac{\omega\lambda\mu_1^2}{\mu_2} \frac{\mu_1^{\frac{1}{2}}\mu_2^{-\frac{1}{2}}}{\lambda^{N+1}} = C(R_0,e,\delta,\kappa,\varepsilon)\frac{\omega\mu_1^\frac{5}{2}\mu_2^{-\frac{3}{2}}}{\lambda^N},\\
\|r^{\tme,5}(t)\|_{L^\infty(\R^2)} &\leq C(\kappa)\frac{1}{\lambda^N}\|\partial_t b_k^1\|_{C^N(\R^2)} \|W_k^{cc,\parallel}\|_{L^\infty(\R^2)}\\
&\leq  C(R_0,u_0,e,\delta,\kappa,\varepsilon,N) \frac{\mu_1^{\frac{1}{2}}\mu_2^{-\frac{1}{2}}}{\lambda^{N+1}},\\
\|r^{\tme,6}(t)\|_{L^\infty(\R^2)}&\leq C(\kappa) \frac{\omega\mu_1}{\mu_2} \|b_k^1\|_{C^{N+1}(\R^2)}\left\|\dv^{-N}A((\varphi')^k_{\mu_1}(\lambda(\cdot -\omega t)),(\Phi'')_{\mu_2}(\lambda\cdot),\xi_k)\right\|_{L^\infty(\R^2)}\\
&\leq  C(R_0,u_0,e,\delta,\kappa,\varepsilon,N) \frac{\omega\mu_1}{\mu_2} \frac{\mu_1^{\frac{1}{2}}\mu_2^{-\frac{1}{2}}}{\lambda^{N+1}} = C(R_0,e,\delta,\kappa,\varepsilon)\frac{\omega\mu_1^\frac{3}{2}\mu_2^{-\frac{3}{2}}}{\lambda^{N+1}},\\
\|r^{\tme,7}(t)\|_{L^\infty(\R^2)} &\leq C(\kappa)\frac{1}{\lambda^N}\|\partial_t b_k^2\|_{C^N(\R^2)} \|W_k^{cc,\perp}\|_{L^\infty(\R^2)}\\
&\leq  C(R_0,u_0,e,\delta,\kappa,\varepsilon,N) \frac{\mu_1^{\frac{1}{2}}\mu_2^{-\frac{1}{2}}}{\lambda^{N+1}},\\
\|r^{\tme,8}(t)\|_{L^\infty(\R^2)}&\leq C(\kappa) \frac{\omega\mu_1}{\mu_2} \|b_k^2\|_{C^{N+1}(\R^2)}\left\|\dv^{-N}B((\varphi')^k_{\mu_1}(\lambda(\cdot -\omega t)),(\Phi'')_{\mu_2}(\lambda\cdot),\xi_k)\right\|_{L^\infty(\R^2)}\\
&\leq  C(R_0,u_0,e,\delta,\kappa,\varepsilon,N) \frac{\omega\mu_1}{\mu_2} \frac{\mu_1^{\frac{1}{2}}\mu_2^{-\frac{1}{2}}}{\lambda^{N+1}} =C(R_0,e,\delta,\kappa,\varepsilon)\frac{\omega\mu_1^\frac{3}{2}\mu_2^{-\frac{3}{2}}}{\lambda^{N+1}},
\end{align*}
and
\begin{align*}
\|\int_0^t\curl r^{\tme,3}(s)\,\ds\|_{H^p(\R^2)}&\leq C(\kappa)\|\partial_t a_k\|_{C^{N+1}(\R^2)}\left\|\nabla \dv^{-N}W_k^c\right\|_{L^\infty(\tor)}\\
 &\leq C(R_0,u_0,e,\delta,\kappa,\varepsilon,N) \frac{\mu_1^\frac{3}{2}\mu_2^{\frac{1}{2}}}{\lambda^{N-1}},\\
 \|\int_0^t\curl r^{\tme,4}(s)\,\ds\|_{H^p(\R^2)}&\leq C(\kappa)\frac{\omega\lambda\mu_1^2}{\mu_2}\|a_k\|_{C^{N+2}(\R^2)}\\
 &\hspace{0,3cm}\cdot\left\|\nabla\dv^{-N}B((\varphi'')^k_{\mu_1}(\lambda(\cdot -\omega t)),(\Phi'')_{\mu_2}(\lambda\cdot),\xi_k)\right\|_{L^\infty(\tor)},\\
&\leq  C(R_0,u_0,e,\delta,\kappa,\varepsilon,N)\frac{\omega\lambda\mu_1^2}{\mu_2} \frac{\mu_1^{\frac{1}{2}}\mu_2^{\frac{1}{2}}}{\lambda^{N}} = C(R_0,e,\delta,\kappa,\varepsilon)\frac{\omega\mu_1^\frac{5}{2}\mu_2^{-\frac{1}{2}}}{\lambda^{N-1}}, \\
\|\int_0^t\curl r^{\tme,5}(s)\,\ds\|_{H^p(\R^2)} &\leq C(\kappa) \|\partial_t b_k^1\|_{C^{N+1}(\R^2)} \|\nabla\dv^{-N}W_k^{cc,\parallel}\|_{L^\infty(\R^2)}\\
&\leq C(R_0,u_0,e,\delta,\kappa,\varepsilon,N) \frac{\mu_1^{\frac{1}{2}}\mu_2^{\frac{1}{2}}}{\lambda^{N}},\\
\|\int_0^t\curl r^{\tme,6}(s)\,\ds\|_{H^p(\R^2)} &\leq C(\kappa) \frac{\omega\mu_1}{\mu_2}\|b_k^1\|_{C^{N+2}(\R^2)}\\
&\hspace{0,3cm}\cdot\left\|\nabla\dv^{-N}A((\varphi')^k_{\mu_1}(\lambda(\cdot -\omega t)),(\Phi'')_{\mu_2}(\lambda\cdot),\xi_k)\right\|_{L^\infty(\R^2)}\\
&\leq C(R_0,u_0,e,\delta,\kappa,\varepsilon,N) \frac{\omega\mu_1}{\mu_2} \frac{\mu_1^{\frac{1}{2}}\mu_2^{\frac{1}{2}}}{\lambda^{N}}=C(R_0,e,\delta,\kappa,\varepsilon)\frac{\omega\mu_1^\frac{3}{2}\mu_2^{-\frac{1}{2}}}{\lambda^{N}},\\
\|\int_0^t\curl r^{\tme,7}(s)\,\ds\|_{H^p(\R^2)} &\leq C(\kappa) \|\partial_t b_k^2\|_{C^{N+1}(\R^2)} \|\nabla\dv^{-N}W_k^{cc,\perp}\|_{L^\infty(\R^2)}\\
&\leq C(R_0,u_0,e,\delta,\kappa,\varepsilon,N) \frac{\mu_1^{\frac{1}{2}}\mu_2^{\frac{1}{2}}}{\lambda^{N}},\\
\|\int_0^t\curl r^{\tme,8}(s)\,\ds\|_{H^p(\R^2)} &\leq C(\kappa) \frac{\omega\mu_1}{\mu_2}\|b_k^2\|_{C^{N+2}(\R^2)}\\
&\hspace{0,3cm}\cdot\left\|\nabla\dv^{-N}B((\varphi')^k_{\mu_1}(\lambda(\cdot -\omega t)),(\Phi'')_{\mu_2}(\lambda\cdot),\xi_k)\right\|_{L^\infty(\R^2)}\\
&\leq C(R_0,u_0,e,\delta,\kappa,\varepsilon,N) \frac{\omega\mu_1}{\mu_2} \frac{\mu_1^{\frac{1}{2}}\mu_2^{\frac{1}{2}}}{\lambda^{N}}=C(R_0,u_0,e,\delta,\kappa,\varepsilon,N)\frac{\omega\mu_1^\frac{3}{2}\mu_2^{-\frac{1}{2}}}{\lambda^{N}}.\qedhere
\end{align*}
\end{proof}

\section{Proof of the main proposition}\label{sec: proof of main prop}
Proposition \ref{prop: main proposition} is proved by choosing all the parameters appropriately, which we do in this section. Let us set 
\begin{itemize}
\item $\mu_1= \lambda^\alpha$,
\item $\mu_2 = \lambda \mu_1= \lambda^{1+\alpha}$,
\item $\omega = \lambda^\beta$
\end{itemize}
for some $\alpha,\beta>0$ to be chosen below. We collect the estimates from Section \ref{sec: estimates perturbations} and \ref{sec: curl estimates} where the parameters $\mu_1,\mu_2$ and $\omega$ need to be balanced.\vspace{0,3cm}\\
\begin{tabular}[h]{l|l|l|l}
Lemma & Term & Order & = \\
\hline
\ref{lemma: estimate correctors} & $u^c$ in $L^2(\R^2)$ & $\mu_1\mu_2^{-1}$ & $\lambda^{-1}$\\
\ref{lemma: estimate correctors}& $u^t$ in $L^2(\R^2)$ & $\omega^{-1}\mu_1^\frac{1}{2}\mu_2^\frac{1}{2}$ & $\lambda^{-\beta +\alpha + \frac{1}{2}}$\\
\ref{lemma: energy inc} & Energy increment & $\mu_1^{-\frac{1}{6}}\mu_2^{-\frac{1}{6}}$ & $\lambda^{-\frac{1}{3}\alpha-\frac{1}{6}}$\\
\ref{lemma: curl w} & $\curl w$ in $H^p(\R^2)$ & $\lambda\mu_1^{\frac{1}{2}-\frac{2}{p}}\mu_2^\frac{3}{2}$ & $\lambda^{\frac{5}{2} + \alpha(2-\frac{2}{p})}$\\

\ref{lemma: curl ut} & $\curl u^t$ in $H^p(\R^2)$ & $\omega^{-1}\lambda\mu_1^{1-\frac{2}{p}}\mu_2^2$ & $\lambda^{3-\beta +\alpha(3-\frac{2}{p})}$\\
\ref{lemma: est rtime} & $R^{\tme}$ in $L^1(\R^2)$ & $\omega\mu_1^\frac{1}{2}\mu_2^{-\frac{3}{2}}$ & $\lambda^{\beta-\alpha-\frac{3}{2}}$\\
\ref{lemma: rquad1N} & $r^{\quadr}$ in $L^\infty(\R^2)$ & $\lambda^{-N}\mu_1\mu_2$ & $\lambda^{2\alpha+ 1 -N}$\\
\ref{lemma: rY} & $r^Y$ in $L^\infty(\R^2)$ & $\omega^{-1}\lambda^{-N}\mu_1\mu_2$ & $\lambda^{-\beta+2\alpha+1-N}$\\
\ref{lemma: rtme} & $r^{\tme}$ in $L^\infty(\R^2)$ & $\lambda^{-N}\mu_1^\frac{1}{2}\mu_2^\frac{1}{2}+\omega\lambda^{-N}\mu_1^\frac{3}{2}\mu_2^{-\frac{1}{2}}$ & $\lambda^{\alpha+\frac{1}{2}-N} + \lambda^{\beta+\alpha-\frac{1}{2}-N}$\\
\ref{lemma: rquad1N} & $\int_0^t\curl r^{\quadr}(s)\,\ds$ in $H^p(\R^2)$ & $\lambda^{1-N}\mu_1\mu_2^2$ & $\lambda^{3\alpha +3 -N}$\\
\ref{lemma: rY} & $\int_0^t\curl r^{Y}(s)\,\ds$ in $H^p(\R^2)$ & $\omega^{-1}\lambda^{1-N}\mu_1\mu_2^2$ & $\lambda^{-\beta +3\alpha +3 -N}$\\
\ref{lemma: rtme} & $\int_0^t\curl r^{\tme}(s)\,\ds$ in $H^p(\R^2)$ & $\lambda^{1-N}\mu_1^\frac{1}{2}\mu_2^\frac{3}{2} + \omega\lambda^{1-N}\mu_1^\frac{3}{2}\mu_2^\frac{1}{2}$ & $\lambda^{2\alpha+\frac{5}{2}-N}$\\
&&& $+\lambda^{\beta+2\alpha+\frac{3}{2}-N}$
\end{tabular}
\vspace{0,3cm}  \\We choose $\alpha,\beta$ and $N$ such that all the exponents in the fourth column of the previous table are negative. This is clear for the first and the third row. Since $2-\frac{2}{p}<0$, we can choose $\alpha \gg 1$ so large such that
\begin{align*}
\frac{5}{2} + \alpha(2-\frac{2}{p}) &< 0
\end{align*}
i.e. we have negative exponents in Line 4. Furthermore, since $3-\frac{2}{p}<1$, let us choose $\alpha$ large enough such that $$3+ \alpha(3-\frac{2}{p})<\alpha+\frac{1}{2}.$$ With this choice of $\alpha$, we only need $\beta$ to satisfy
\begin{align*}
3+\alpha(3-\frac{2}{p})<\alpha + \frac{1}{2} <\beta < \alpha +\frac{3}{2}.
\end{align*}
With such a $\beta$, Line 1 -- 6 have negative exponents of $\lambda$. Having $\alpha$ and $\beta$ fixed, it only remains to choose $N$. Since $N$ enters all the remaining exponents with a negative sign, we can simply pick $N\in\N$ large enough such that all exponents are negative.
Let 
\begin{align*}
\gamma_0 = \text{ exponent in the table with the smallest magnitude}
\end{align*}
which satisfies $\gamma_0 <0$ by our choice of $\alpha,\beta, N$.
We can now verify the claims of Proposition \ref{prop: main proposition}. For $(i)$, we have by \eqref{energy increment}
\begin{align*}
\left|e(t) \left(1-\frac{\delta}{2}\right) - \int_{\R^2} |u_1|^2\,\dx\right|&<  \frac{1}{8}\delta e(t) +  C(R_0,r_0, u_0,e, \delta,\kappa,\varepsilon)\lambda^{\gamma_0}
\end{align*}
and we can choose $\lambda$ large enough such that $(i)$ is satisfied. For $(v)$, we use Lemma \ref{lemma: estimate up},  \mbox{Lemma \ref{lemma: estimate correctors}} and Lemma \ref{lemma: v in L2}
\begin{align*}
\|(u_1-u_0)(t)\|_{L^2(\R^2)}&\leq \|u^p(t)\|_{L^2(\R^2)} + \|u^c(t)\|_{L^2(\R^2)}  + \|u^t(t)\|_{L^2(\R^2)} + \|v(t)\|_{L^2(\R^2)}\\
&\leq10 \delta^\frac{1}{2}+\frac{ C(\kappa,\varepsilon)}{\lambda^\frac{1}{2}} + C(R_0,u_0,e,\delta,\kappa,\varepsilon)\lambda^{\gamma_0} + \|r_0\|_{C_tL^2_x}
\end{align*}
Using that $\|r_0\|_{C_tL^2_x}\leq \frac{1}{32}\delta$ by assumption, we can choose $\lambda$ large enough such that 
\begin{align*}
\|u_1-u_0(t)\|_{L^2(\R^2)}&\leq 11 \delta^\frac{1}{2} ,
\end{align*}
i.e. $(v)$ is satisfied with $M_0 = 11$.
For $(vi)$, we use Lemmas \ref{lemma: curl w}, \ref{lemma: curl ut} and \ref{lemma: curl v}
\begin{align*}
\|\curl (u_1 - u_0)(t)\|^p_{H^p(\R^2)} &\leq \|\curl w (t)\|^p_{H^p(\R^2)} + \|\curl u^t(t)\|^p_{H^p(\R^2)} + \|\curl v(t)\|^p_{H^p(\R^2)}\\
&\leq C(R_0,u_0,e,\delta,\kappa,\varepsilon)\lambda^{p\gamma_0} + \|\int_0^t \curl r_0(s)\,\ds\|^p_{H^p(\R^2)} 
\end{align*}
and we can choose $\lambda$ large enough such that $(vi)$ is satisfied. For $(iv)$, we have by Lemma \ref{lemma: est rlin},  \ref{lemma: est rlin2},  \ref{lemma: est rlin3}, \ref{lemma: est rkappa},  \ref{lemma: est rquadr}, \ref{lemma: est ry} and \ref{lemma: est rtime}
\begin{align*}
\|R_1(t)\|_{L^1(\R^2)}&\leq \|R^{\operatorname{lin},1}(t)\|_{L^1(\R^2)} + \|R^{\operatorname{lin},2}(t)\|_{L^1(\R^2)} + \|R^{\operatorname{lin},3}(t)\|_{L^1(\R^2)}
\\
 &\hspace{0,3cm} +\|R^\kappa(t)\|_{L^1(\R^2)} +  \|R^{\quadr}(t)\|_{L^1(\R^2)} + \|R^Y(t)\|_{L^1(\R^2)}  + \|R^{\tme}(t)\|_{L^1(\R^2)}\\
&\leq \frac{\eta}{2} +4\|r_0\|_{C_tL_x^2}^2 + 2\|r_0\|_{C_tL_x^2} \|u_0(t)\|_{L^2(\R^2)} +C(R_0,u_0,e,\delta,\kappa,\varepsilon,N) \lambda^{\gamma_0}.
\end{align*}
Noting that $\|r_0(t)\|_{L^2(\R^2)}^2\leq \|r_0(t)\|_{L^2(\R^2)}$ since $\|r_0(t)\|_{L^2(\R^2)}\leq 1$ by assumption, we can choose $\lambda$ large enough to obtain $(iv)$. For $(ii)$, we have because of the compact support of $r_1$
\begin{align*}
\|r_1(t)\|_{L^2(\R^2)} &\leq C(\kappa) \|r_1(t)\|_{L^\infty(\R^2)} \\
&\leq C(\kappa)\left(\|r^{\quadr}(t)\|_{L^\infty(\R^2)} + \|r^Y(t)\|_{L^\infty(\R^2)} + \|r^{\tme}(t)\|_{L^\infty(\R^2)}\right)\\
& \leq C(R_0,u_0,e,\delta,\kappa,\varepsilon, N)\lambda^{\gamma_0},
\end{align*}
and by the previous estimate on $\|u_1(t)-u_0(t)\|_{L^2(\R^2)}$
\begin{align*}
\|u_1\|_{C_tL^2_x} \|r_1(t)\|_{L^2(\R^2)}&\leq \|u_0\|_{C_tL^2_x} \|r_1(t)\|_{L^2(\R^2)} + \|u_1-u_0\|_{C_tL^2_x} \|r_1(t)\|_{L^2(\R^2)}\\
&\leq C(R_0,r_0,u_0,\delta,\kappa,\varepsilon, N)\lambda^{\gamma_0} .
\end{align*}
Finally, we also have
\begin{align*}
\|\int_0^t\curl r_1(s)\,\ds\|^p_{H^p(\R^2)} &\leq \|\int_0^t\curl r^{\quadr}(s)\,\ds\|^p_{H^p(\R^2)} + \|\int_0^t\curl r^Y(s)\,\ds\|^p_{H^p(\R^2)} \\
&\hspace{0,3cm}+ \|\int_0^t\curl r^{\tme}(s)\,\ds\|^p_{H^p(\R^2)}\\
&\hspace{0,3cm} \leq C(R_0,u_0,e,\delta,\kappa,\varepsilon,N) \lambda^{p\gamma_0}.
\end{align*}
Again, $\lambda$ can be chosen large enough such that $(iii)$ is satisfied. Assume we have given two energy profiles $e_1,e_2$ with $e_1=e_2$ on $[0,t_0]$ for some $t_0\in[0,1]$.  The values that we add with $w(t), u_c(t), u^t(t)$ depend only on pointwise (in time) values of the previous steps, while $v(t)$ depends only on values of the previous steps on $[0,t]$. Therefore, one can do the construction for $e_1$ and $e_2$ simultaneously, choosing the same values for all the parameters in each iteration step, thereby producing two solutions $u_1$, $u_2$ to \eqref{2D Euler} that satisfy $u_1=u_2$ on $[0,t_0]$.

\nocite{*}
\bibliographystyle{alpha}
\bibliography{bibliography}

\begin{thebibliography}{BMNV23}

\bibitem[BC21]{brue2021nonuniqueness}
Elia Bru{\`e} and Maria Colombo.
\newblock Nonuniqueness of solutions to the euler equations with vorticity in a
  lorentz space.
\newblock {\em arXiv preprint arXiv:2108.09469}, 2021.

\bibitem[BCDL21]{brue2021positive}
Elia Bru{\`e}, Maria Colombo, and Camillo De~Lellis.
\newblock Positive solutions of transport equations and classical nonuniqueness
  of characteristic curves.
\newblock {\em Archive for Rational Mechanics and Analysis}, 240(2):1055--1090,
  2021.

\bibitem[BLSV18]{buckmaster2018}
Tristan Buckmaster, Camillo Lellis, Jr~Székelyhidi, and Vlad Vicol.
\newblock Onsager's conjecture for admissible weak solutions.
\newblock {\em Comm. Pure Appl. Mat}, 72(2):229--274, 2018.

\bibitem[BMNV23]{novack1}
Tristan Buckmaster, Nader Masmoudi, Matthew Novack, and Vlad Vicol.
\newblock Intermittent convex integration for the 3d euler equations.
\newblock {\em Ann. of Math. Studies}, 217, 2023.

\bibitem[BMS21]{burczak-mod-sze21}
Jan Burczak, Stefano Modena, and L\'aszl\'o Székelyhidi.
\newblock Non uniqueness of power-law flows.
\newblock {\em Communications in Mathematical Physics}, 388(1):199--243, 2021.

\bibitem[BS20]{bressan2020posteriori}
Alberto Bressan and Wen Shen.
\newblock A posteriori error estimates for self-similar solutions to the euler
  equations.
\newblock {\em arXiv preprint arXiv:2002.01962}, 2020.

\bibitem[BV19]{buckmaster2019nonuniqueness}
Tristan Buckmaster and Vlad Vicol.
\newblock Nonuniqueness of weak solutions to the navier-stokes equation.
\newblock {\em Annals of Mathematics}, 189(1):101--144, 2019.

\bibitem[CL21]{cheskidov2021nonuniqueness}
Alexey Cheskidov and Xiaoyutao Luo.
\newblock Nonuniqueness of weak solutions for the transport equation at
  critical space regularity.
\newblock {\em Annals of PDE}, 7(2):1--45, 2021.

\bibitem[CL22]{cheskidov2022extreme}
Alexey Cheskidov and Xiaoyutao Luo.
\newblock Extreme temporal intermittency in the linear sobolev transport:
  almost smooth nonunique solutions.
\newblock {\em arXiv preprint arXiv:2204.08950}, 2022.

\bibitem[DLS13]{de2013dissipative}
Camillo De~Lellis and L{\'a}szl{\'o} Sz{\'e}kelyhidi.
\newblock Dissipative continuous euler flows.
\newblock {\em Inventiones mathematicae}, 193:377--407, 2013.

\bibitem[DLSJ09]{de2009euler}
Camillo De~Lellis and L{\'a}szl{\'o} Sz{\'e}kelyhidi~Jr.
\newblock The euler equations as a differential inclusion.
\newblock {\em Annals of mathematics}, 2(3):1417--1436, 2009.

\bibitem[DLSJ14]{de2014dissipative}
Camillo De~Lellis and L{\'a}szl{\'o} Sz{\'e}kelyhidi~Jr.
\newblock Dissipative euler flows and onsager's conjecture.
\newblock {\em Journal of the European Mathematical Society}, 16(7):1467--1505,
  2014.

\bibitem[GKN23a]{giri20233}
Vikram Giri, Hyunju Kwon, and Matthew Novack.
\newblock The {$ L^3$}-based strong onsager theorem.
\newblock {\em arXiv preprint arXiv:2305.18509}, 2023.

\bibitem[GKN23b]{giri2023wavelet}
Vikram Giri, Hyunju Kwon, and Matthew Novack.
\newblock A wavelet-inspired {$ L^3$}-based convex integration framework for
  the euler equations.
\newblock {\em arXiv preprint arXiv:2305.18142}, 2023.

\bibitem[GS21]{giri2021non}
Vikram Giri and Massimo Sorella.
\newblock Non-uniqueness of integral curves for autonomous hamiltonian vector
  fields.
\newblock {\em arXiv preprint arXiv:2108.05050}, 2021.

\bibitem[Ise18]{isett2018proof}
Philip Isett.
\newblock A proof of onsager's conjecture.
\newblock {\em Annals of Mathematics}, 188(3):871--963, 2018.

\bibitem[Loe06]{loeper2006uniqueness}
Gr{\'e}goire Loeper.
\newblock Uniqueness of the solution to the vlasov--poisson system with bounded
  density.
\newblock {\em Journal de math{\'e}matiques pures et appliqu{\'e}es},
  86(1):68--79, 2006.

\bibitem[Men23]{mengual2023non}
Francisco Mengual.
\newblock Non-uniqueness of admissible solutions for the 2d euler equation with
  {$L^p$} vortex data.
\newblock {\em arXiv preprint arXiv:2304.09578}, 2023.

\bibitem[MS18]{modena2018non}
Stefano Modena and L{\'a}szl{\'o} Sz{\'e}kelyhidi.
\newblock Non-uniqueness for the transport equation with sobolev vector fields.
\newblock {\em Annals of PDE}, 4(2), 2018.
\newblock 18.

\bibitem[MS19]{modena2019non}
Stefano Modena and L{\'a}szl{\'o} Sz{\'e}kelyhidi.
\newblock Non-renormalized solutions to the continuity equation.
\newblock {\em Calculus of Variations and Partial Differential Equations}, 58,
  2019.
\newblock 208.

\bibitem[MS20]{modena2020convex}
Stefano Modena and Gabriel Sattig.
\newblock Convex integration solutions to the transport equation with full
  dimensional concentration.
\newblock {\em Annales de l'Institut Henri Poincar{\'e} C, Analyse non
  lin{\'e}aire}, 37:1075--1108, 2020.

\bibitem[PS21]{pitcho2021almost}
Jules Pitcho and Massimo Sorella.
\newblock Almost everywhere non-uniqueness of integral curves for
  divergence-free sobolev vector fields.
\newblock {\em arXiv preprint arXiv:2108.03194}, 2021.

\bibitem[SM93]{stein1993harmonic}
Elias~M Stein and Timothy~S Murphy.
\newblock {\em Harmonic analysis: real-variable methods, orthogonality, and
  oscillatory integrals}, volume~3.
\newblock Princeton University Press, 1993.

\bibitem[Vis18a]{vishik2018instability1}
Misha Vishik.
\newblock Instability and non-uniqueness in the cauchy problem for the euler
  equations of an ideal incompressible fluid. part i.
\newblock {\em arXiv preprint arXiv:1805.09426}, 2018.

\bibitem[Vis18b]{vishik2018instability2}
Misha Vishik.
\newblock Instability and non-uniqueness in the cauchy problem for the euler
  equations of an ideal incompressible fluid. part ii.
\newblock {\em arXiv preprint arXiv:1805.09440}, 2018.

\bibitem[Yud62]{yudovich1962some}
Victor~Iosifovich Yudovich.
\newblock Some bounds for solutions of elliptic equations.
\newblock {\em Matematicheskii Sbornik}, 101:229--244, 1962.

\bibitem[Yud63]{yudovich1963non}
Victor~Iosifovich Yudovich.
\newblock Non-stationary flows of an ideal incompressible fluid.
\newblock {\em Zhurnal Vychislitel'noi Matematiki i Matematicheskoi Fiziki},
  3(6):1032--1066, 1963.

\end{thebibliography}
\end{document}